\documentclass[11pt]{article}

\usepackage[T1]{fontenc}
\usepackage[utf8]{inputenc}
\usepackage[british]{babel}
\usepackage[
rm={oldstyle=false,proportional=true},
sf={oldstyle=false,proportional=true},
tt={oldstyle=false,proportional=true}
]{cfr-lm}
\usepackage[babel]{microtype}

\usepackage{complexity}

\usepackage{authblk}

\usepackage{mathtools}
\usepackage{amssymb}
\usepackage{amsthm}

\usepackage{makecell}

\usepackage{tikz}
\tikzset{blue node/.style={circle,fill=blue!80,draw,minimum size=0.4cm,inner
sep=0pt}, }
\tikzset{red node/.style={circle,fill=red!80,draw,minimum size=0.4cm,inner
sep=0pt}, }
\tikzset{unc node/.style={circle,fill=gray!20,draw,minimum size=0.4cm,inner
sep=0pt}, }

\usepackage[misc]{ifsym}

\usepackage{url}
\usepackage[noadjust]{cite}
\usepackage[hidelinks]{hyperref}
\usepackage[capitalise, noabbrev]{cleveref}

\usepackage{orcidlink}

\newtheorem{theorem}{Theorem}[section]
\newtheorem{lemma}[theorem]{Lemma}
\newtheorem{corollary}[theorem]{Corollary}

\theoremstyle{definition}
\newtheorem{definition}[theorem]{Definition}
\newtheorem{problem}[theorem]{Open problem}

\DeclareMathOperator{\dist}{d}

\newcommand{\authDetails}[3]{\author{\mbox{#1\,$^{\textrm{\href{mailto:#2}{\Letter}}\,\,\raisebox{-0.2ex}{\orcidlink{#3}}\,}$}}}

\AtBeginDocument{%
  \DeclareFontShape{T1}{clm2}{m}{scit}{<->ssub*clm2/m/scsl}{}%
}

\begin{document}

\title{On graph automorphisms related to \textsc{snort}}

\authDetails{Rylo Ashmore}{rashmore@mun.ca}{0009-0006-0728-0860}
\authDetails{Beth Ann Austin}{eaustin@mun.ca}{0009-0000-2805-6303}
\authDetails{Alfie M. Davies}{research@alfied.xyz}{0000-0002-4215-7343}
\authDetails{Danny Dyer}{dyer@mun.ca}{0000-0001-6921-1517}
\authDetails{William Kellough}{wskellough@mun.ca}{0009-0008-1746-3251}

\date{}
\affil{
    \small{
        Department of Mathematics and Statistics,\\
        Memorial University of Newfoundland,\\
        Canada
    }
}

\maketitle

\begin{abstract}
    \noindent
    We study the outcomes of various positions of the game \textsc{snort}. When
    played on graphs admitting an automorphism of order two that maps vertices
    outside of their closed neighbourhoods (called \emph{opposable} graphs),
    the second player has a winning strategy. We give a necessary and
    sufficient condition for a graph to be opposable, and prove that the
    property of being opposable is preserved by several graph products. We show
    examples that a graph being second-player win does not imply that the graph
    is opposable, which answers Kakihara's conjecture. We give an analogous
    definition to opposability, which gives a first-player winning strategy; we
    prove a necessary condition for this property to be preserved by the
    Cartesian and strong products. As an application of our results, we
    determine the outcome of \textsc{snort} when played on various $n\times m$
    chess graphs.
\end{abstract}

\section{Introduction}

Two farmers share a number of fields. One farmer would like to use the fields
for their bulls, and the other for their cows. They both agree that no bulls
should be adjacent to any cows. They decide that they will take turns in
picking fields. This is the game of \textsc{snort}, devised by Simon Norton,
and first written about in \cite[pp.~91--92]{conway:onag} and
\cite[pp.~145--147]{berlekamp.conway.guy:wwv1} in roughly the same way we have
just described it. Being competitive farmers (or perhaps unfriendly), they each
want to make sure that they have the most  fields, restricting the choice of
the other farmer. As such, if a farmer, on their turn, has no field to pick
(either because every field has been claimed already, or because picking an
unclaimed field would place bulls next to cows), then they are said to lose. To
the initiated reader, this is the \emph{normal play} convention in
Combinatorial Game Theory. This is evidently a two-player, combinatorial game.

We will not talk of bulls, cows, and fields when we play this game. Instead, as
is typically done, we will consider a game of \textsc{snort} to be played on a
finite, simple graph. Enforcing the graph to be simple is reasonable, since
adding loops or multiple edges would not affect the play of the game in any
way, as the reader may confirm to themselves. Enforcing finiteness, however,
does change things; the brave reader may wish to investigate the transfinite
case.

The two players, Left and Right, will alternate between colouring uncoloured
vertices blue and red respectively. We say that a vertex is \emph{tinted} blue
(respectively red) if it is adjacent to a blue (respectively red) vertex but is
not itself coloured. Since Left colours vertices blue, Left is never allowed to
colour a vertex that is tinted red. Similarly, Right is never allowed to colour
a vertex that is tinted blue. The first player who is unable to colour a vertex
on their turn loses the game. Note that we assume all vertices of a graph start
untinted unless otherwise specified.

The attentive reader may ask whether the graph should be planar, given that
this would be the layout of the fields in the original description, and indeed
this is how the game was first played and analyzed (see, for example,
\cite[p.~91]{conway:onag}). As such, the ruleset we have just described should
be considered a generalization of \textsc{snort}, but it is a standard one.

In this paper, we analyze the outcomes of \textsc{snort}, starting with graphs
in which the second player has a winning strategy. To do this, we make use of a
structural property of graphs, introduced by Kakihara \cite{kakihara:snort}, to
give a player a strategy in which they can ``mirror'' their opponent's previous
move (much in the same vein as a typical \emph{Tweedledum and Tweedledee}
argument \cite[p.~3]{berlekamp.conway.guy:wwv1}). A graph $G$ is
\emph{opposable} if it contains an automorphism $f$ of order two such that
$f(v)\notin N[v]$ for all $v\in V(G)$. We will say that a graph is
\emph{non-opposable} if it is not opposable. Such an automorphism, in this case
$f$, is called an \emph{opposition} of $G$. It is known that the second player
wins \textsc{snort} on an opposable graph (see \cite[Proposition 3.6 on
p.~24]{kakihara:snort} or \cref{sec: opposable graphs} for a proof). Kakihara
\cite{kakihara:snort} and Arroyo \cite{arroyo:dawsons} both independently
studied such graphs, finding various nice families. Among his results, Arroyo
\cite{arroyo:dawsons} showed that the property of being opposable is preserved
by the Cartesian product. In \cref{sec: opposable graphs}, we further develop
Arroyo and Kakihara's works by finding new properties and behaviours of
opposable graphs, including a generalization of Arroyo's graph product result.

Although opposability necessarily implies that the second player wins, it is
unknown exactly how having a winning strategy for the second player and being
opposable differ. That is, for which non-opposable graphs does the second
player win? Kakihara conjectured in \cite{kakihara:snort} that a graph is
opposable if and only if the second player has a winning strategy. In
\cref{sec: non-opposable 2nd player win graphs}, we show that this conjecture
is false, and do so by giving infinite families of counterexamples.

Arroyo \cite{arroyo:dawsons} also defined a parallel concept to opposability,
trying to begin to classify graphs for which the first player wins
\textsc{snort}. In \cref{sec: almost opposable graphs}, we generalize this
definition to what we call \emph{almost opposable graphs}, and explore how this
property behaves with respect to various graph products.

In \cref{sec: peaceable queens game}, we demonstrate an application of our
results to a natural family of games involving placing chess pieces onto a
chessboard such that no two pieces can attack each other. These types of chess
problems have been studied for centuries, particularly for placing queens. For
a list of results, a list of real-world applications, and a summary of the
history, we direct the reader to the excellent survey by Bell and Stevens
\cite{bell.stevens:survey}. The sequence-obsessed reader may be aware that the
$n$th number in the sequence A250000 in the OEIS is the maximum value $m$ such
that $m$ white queens and $m$ black queens can be placed on an $n\times n$
chessboard such that no two queens of opposite colour can attack each other.
See \cite{sloane:on-line} for some relevant results on sequence A250000, and
\cite{yao.zeilberger:numerical} for a study on a continuous version of the
sequence. We consider a gamified, adversarial version of the sequence A250000,
which we call the Peaceable Queens Game, where two players are each assigned a
colour and take turns placing a queen of their assigned colour on a square of
an $n\times m$ chessboard. This turns out to be a special case of
\textsc{snort}. We determine the outcome of the Peaceable Queens Game for
chessboards with at least one odd side. We also determine the outcomes of
analogous peaceable games played with other chess pieces.

Finally, in \cref{sec: further directions}, we end with some open questions and
directions for future work. For graph theorists unfamiliar with combinatorial
game theory, we suggest Siegel's textbook \cite{siegel:combinatorial}.

\section{Opposable Graphs}
\label{sec: opposable graphs}

For completeness, we begin with Kakihara's proof \cite{kakihara:snort} that the
second player wins \textsc{snort} on opposable graphs---note that Kakihara in
fact cites a preprint by Stacey Stokes and Mark D.\ Schlatter, but that work
does not seem to have appeared anywhere, and the authors could not be reached.
The technique used in this proof will be utilized throughout the paper.

\begin{lemma}[{\cite[Proposition 3.6 on p.~24]{kakihara:snort}}]
    \label{lem: opposable implies second player wins snort}
    If $G$ is a graph that admits an opposition $f$, then playing
    \textsc{snort} on $G$ is second-player win.
\end{lemma}

\begin{proof}
    We will show by induction that a winning strategy for the second player is
    to always respond with $f(v)$ when the first player plays on $v$. On the
    first move, since $f$ is an opposition, the second player responding by
    colouring $f(v)$ is a valid move.

    Suppose that play continues in this way, but that at some point, when the
    first player plays to some $u$, the second player cannot respond to $f(u)$.
    The only way that the second player would be unable to play $f(u)$ is if
    either $f(u)$ is already coloured, or $f(u)$ is adjacent to some vertex $w$
    that the first player already claimed.

    Suppose first that $f(u)$ is already coloured. If $f(u)$ was coloured by
    the second player, then, since $f$ is an automorphism, it must be because
    the first player played $u$ at some earlier point in the game. However,
    this contradicts the fact that $u$ was just played by the first player.
    Otherwise, if $f(u)$ was coloured by the first player, then the second
    player would have responded at the time on $f(f(u)) = u$. This contradicts
    the fact that the first player just played on $u$.

    It must therefore be the case that $f(u)$ is adjacent to some vertex $w$
    that the first player already claimed. When the first player played on $w$,
    the second player would have responded to $f(w)$. Since $f$ is an
    automorphism and $\{f(u),w\}$ is an edge, $\{f(f(u)),f(w)\}$ is an edge.
    Since $f$ has order two, we obtain that $\{u,f(w)\}$ is an edge. This
    contradicts the fact that the first player just played on $u$ since the
    first player's move would be adjacent to the earlier move $f(w)$ by the
    second player.

    Thus, whenever the first player makes a move, the second player will be
    able to respond. So the player to run out of moves first will be the first
    player. Therefore $G$ is second-player win.
\end{proof}

Next, we establish a way of showing that graphs that partition into opposable
graphs are themselves opposable. To do this, we introduce a definition
involving matchings. We will see in \cref{thm:matchingCharacterization} that
this new definition gives a structural characterization of graphs that admit an
opposition.

\begin{definition}
    A perfect matching $M$ is an \emph{opposition matching} if for any two
    edges $v_1v_2\neq v_3v_4$ in $M$, the graph $H$ induced by
    $\{v_1,v_2,v_3,v_4\}$ is isomorphic to either $K_4$, $C_4$, or just the two
    edges of $M$.
\end{definition}

\begin{theorem}\label{thm:matchingCharacterization}
    A graph $G$ has an opposition if and only if the complement $\overline{G}$
    has an opposition matching.
\end{theorem}

\begin{proof}
    Let $G$ have an opposition, say $f$. Then in $\overline{G}$ we construct a
    matching $M$ by matching $v$ to $f(v)$. Since $f$ has order two and no
    fixed points, this is well-defined. Since oppositions match vertices to
    non-adjacent vertices, they must be adjacent in the complement, and thus
    are a valid edge of a matching. Since $f$ is an automorphism, we get that
    it is a bijection, and thus $M$ is a perfect matching. Let $v_1v_2$ and
    $v_3v_4$ be two edges of $M$.

    If $v_1v_3$ is an edge in $G$, then so too is $v_2v_4$ because $f$ is an
    automorphism. By the contrapositive, if $v_2v_4$ is an edge in
    $\overline{G}$, then $v_1v_3$ is an edge in $\overline{G}$. Because $f$ is
    an order two automorphism, this is an if and only if argument. The edge
    pair $v_1v_3$ and $v_2v_4$ either both exist, or both do not exist.
    Similarly, the edge pair $v_1v_4$ and $v_2v_3$ either both exist, or both
    do not exist. If no edge pairs exist in $\overline{G}$, then the subgraph
    $H$ induced by $\{v_1,v_2,v_3,v_4\}$ is just $v_1v_2+v_3v_4$. If exactly
    one edge pair exists in $\overline{G}$, then $H \cong C_4$, and if both
    edge pairs exist in $\overline{G}$, then $H \cong K_4$.

    Conversely, suppose $\overline{G}$ has an opposition matching $M$. Define a
    function $f:V(G)\to V(G)$ by mapping $v\in V(G)$ to the matched vertex in
    $M$. We claim $f$ is an opposition in $G$. Let $v_1v_2$ be an edge of $G$,
    and let $f(v_1)f(v_2)=v_3v_4$. Then $v_1v_3$ and $v_2v_4$ are edges in $M$.
    Since $M$ is an opposition matching, the subgraph $H$ induced in
    $\overline{G}$ is isomorphic to either $K_4$, $C_4$, or just the two edges. 
    \begin{itemize}
        \item
            If $H\cong K_4$, then $v_1v_2\not\in E(G)$, and this contradicts
            $v_1v_2\in E(G)$ assumed.
        \item
            If $H\cong C_4$, then $v_1v_2\in E(G)$ implies that the two edges
            not in $H$ are $v_1v_2$ and $v_3v_4$. Thus $v_3v_4$ is an edge of
            $G$.
        \item
            If $H$ is just $v_1v_3 + v_2v_4$, then $v_3v_4$ is not an edge of
            $H$, and so $v_3v_4\in E(G)$.
    \end{itemize}
    It follows that $f$ is an edge-preserving function. Furthermore, $f$ is a
    bijection because it was constructed through a perfect matching. Thus $f$
    is an automorphism. Because $M$ is a perfect matching, the function $f$ has
    order two and has no fixed points. Since $f$ is based on a perfect
    matching, the edge $vf(v)$ always exists in $\overline{G}$, and thus does
    not exist in $G$. Thus the automorphism does not map vertices to adjacent
    vertices. So $f$ is an opposition of $G$.
\end{proof}

Now we can reason about graphs that nicely partition. We use $2K_2$ to denote
the graph on 4 vertices made up of two disjoint edges and $4K_1$ to be the
empty graph on 4 vertices.

\begin{theorem}
    \label{thm: decomposition?}
    If $G$ can be vertex-partitioned into sets $V_1,\dots, V_{k}$ such that the
    subgraph induced by $V_i$ is isomorphic to an opposable graph $H_i$, and
    for any pair of paired vertices, $\{u_i,v_i\}\subseteq V_i$ and
    $\{u_j,v_j\}\subseteq V_j$, the set $\{u_i, v_i, u_j, v_j\}$ induces either
    a $C_4$, a $2K_2$, or a $4K_1$, then $G$ is opposable.
\end{theorem}

\begin{proof}
    Let $G$ be partitioned into $V_i$ and let $f_i:V_i\rightarrow V(H_i)$ be an
    isomorphism between the subgraph induced by $V_i$ and $H_i$. Let
    $\varphi_i:V(H_i)\rightarrow V(H_i)$ be an opposition on $H_i$. We create a
    function $\psi:V(G)\rightarrow V(G)$ by mapping $v\in V_i$ by
    $\psi(v)=f_i^{-1}(\varphi_i(f_i(v)))$ for each $i\in \{1,\dots, k\}$ and
    claim it is an opposition. Note that $\psi$ has no fixed points since
    $\varphi_i$ has no fixed points and $f_i$ is an isomorphism. It also
    follows from the construction that $\psi$ has order two, and does not map
    vertices to adjacent vertices, otherwise $\varphi_i$ would not be an
    opposition on $H_i$.

    It remains to show that $\psi$ preserves edges so that it is a proper
    automorphism. Since $\varphi_i$ is an automorphism, we find that edges
    within $V_i$ are preserved. If two vertices $v_i \in V_i$ and $v_j \in V_j$
    with $V_i \neq V_j$ are not adjacent (and $v_i\psi(v_j)\not\in E(G)$), then
    $\{v_i, v_j, \psi(v_i), \psi(v_j)\}$ induce a $4K_1$. Suppose instead that
    $v_iv_j$ is an edge between $V_i$ and $V_j$, with $i \neq j$. Let $u_i,u_j$
    be the corresponding paired vertices. Then, by the assumptions made in the
    statement of the theorem, $\{u_i,v_i,u_j,v_j\}$ induce either a $C_4$ or
    $2K_2$ (a $4K_1$ is not possible as the edge $v_iv_j$ is assumed to exist),
    and in either case there must be an edge $u_iu_j$.
\end{proof}

\begin{definition}
    Let $\varphi(g_1,g_2,h_1,h_2)$ be any logical formula consisting of atomics
    $g_1g_2\in E(G)$, $h_1h_2\in E(H)$, $g_1=g_2$, $h_1=h_2$, and connectives
    $\wedge$, $\vee$, $\neg$, and $\Rightarrow$. We abuse graph product
    notation and define $G\varphi H$ to be the graph on vertices $V(G)\times
    V(H)$ where $(g_1,h_1)(g_2,h_2)\in E(G\varphi H)$ if and only if
    $\varphi(g_1,g_2,h_1,h_2)$.
\end{definition}

We provide some definitions of common graph products in terms of this notation.
\begin{center}
    \begin{tabular}{|c|c|c|}
    \hline
    \makecell{Name} & \makecell{Symbol} & \makecell{Corresponding
    $\varphi(g_1,g_2,h_1,h_2)$}\\\hline\hline

    \makecell{Cartesian} & \makecell{$G\square H$} & \makecell{$(g_1=g_2\wedge
    h_1h_2\in E(H))\vee$\\$(g_1g_2\in E(G)\wedge h_1=h_2)$}\\\hline

    \makecell{Strong} & \makecell{$G\boxtimes H$} & \makecell{$(g_1=g_2\wedge
    h_1h_2\in E(H))\vee$\\$(g_1g_2\in E(G)\wedge h_1=h_2)\vee$ \\$(g_1g_2\in
    E(G)\wedge h_1h_2\in E(H))$}\\\hline

    \makecell{Tensor} & \makecell{$G\times H$} & \makecell{$g_1g_2\in
    E(G)\wedge h_1h_2\in E(H)$}\\\hline

    \makecell{Lexicographic} & \makecell{$G\bullet H$} & \makecell{$g_1g_2\in
    E(G)\vee\left(g_1=g_2\wedge h_1h_2\in E(H)\right)$}\\\hline

    \makecell{Co-normal} & \makecell{$G*H$}& \makecell{$g_1g_2\in E(G)\vee
    h_1h_2\in E(H)$}\\\hline

    \makecell{Homomorphic} & \makecell{$G\ltimes H$} & \makecell{$g_1=g_2
    \vee$\\ $\left( g_1g_2\in E(G) \wedge \neg \left( h_1h_2\in E(H) \right)
    \right)$}\\\hline
    
    \makecell{Cihpromomoh} & \makecell{$G\rtimes H$} & \makecell{$h_1=h_2 \vee$\\
    $\left( h_1h_2\in E(H) \wedge \neg \left( g_1g_2\in E(G) \right)
    \right)$}\\\hline
    \end{tabular}
\end{center}

For a given graph product $\varphi$ and vertex $v\in V(H)$, the graph $G.v
\subseteq G\varphi H$ is defined as the subgraph induced by the set of vertices
$\{(u,v)\mid u\in V(G)\}$. The subgraph $u.H$ is defined similarly. For more on
graph products we direct the reader to \cite{imrich.klavzar:product}.

We can use \cref{thm: decomposition?} to prove that graph products preserve
opposability. However, it is cleaner to show this by working directly from the
definition of a graph product.

\begin{theorem}
    \label{thm:graphProdOpposableLogic}
    If $\varphi$ satisfies 
    \[
        ((h_1=h_2)\wedge\neg(g_1g_2\in E(G)))\Rightarrow\neg
        \varphi(g_1,g_2,h_1,h_2)
    \]
    and $G$ is opposable, then $G\varphi H$ is opposable.
\end{theorem}

\begin{proof}
    Let $f$ be an opposition of $G$. Define $\hat{f}:V(G\varphi H)\rightarrow
    V(G\varphi H)$ by $\hat{f}((g,h))=(f(g),h)$. Since $f$ is an order two
    bijection with no fixed points, $\hat{f}$ is also an order two bijection
    with no fixed points. Let $(g,h)\in V(G\varphi H)$. Then
    $\hat{f}((g,h))=(f(g),h)$. Since $gf(g)\not\in E(G)$ and $h=h$, we obtain
    that $\neg \varphi(g,f(g),h,h)$, and so $(g,h) \not\sim_{G\varphi H}
    \hat{f}((g,h))$. Finally, let $(g_1,h_1)(g_2,h_2)$ be an edge in $G\varphi
    H$. Note that $h_1,h_2$ are unchanged by $\hat{f}$, and $\hat{f}$ changes
    $g_1$ and $g_2$ to $f(g_1)$ and $f(g_2)$ respectively. Since $f$ is an
    automorphism, $g_1=g_2$ if and only if $f(g_1)=f(g_2)$ and $g_1g_2\in E(G)$
    if and only if $f(g_1)f(g_2)\in E(G)$. Thus $\varphi(g_1,g_2,h_1,h_2)$ is
    satisfied if and only if $\varphi(f(g_1),f(g_2),h_1,h_2)$ is satisfied.
    Since by assumption $(g_1,h_1)(g_2,h_2)\in E(G\varphi H)$, we obtain that
    $\hat{f}$ is a homomorphism. Thus $\hat{f}$ is an opposition, and $G\varphi
    H$ is opposable.
\end{proof}

\begin{corollary}
    \label{cor: graph products preserve opposable}
    If $G$ is opposable and $H$ is any graph, then the Cartesian product
    $G\square H$, the strong product $G\boxtimes H$, the tensor product
    $G\times H$, the lexicographic product $G\bullet H$, the co-normal product
    $G*H$, and the homomorphic product $G\ltimes H$ are all opposable.
\end{corollary}

\begin{proof}
    The proof is by checking the definitions of each of the graph products
    against the condition for $\varphi$ in order to apply Theorem
    \ref{thm:graphProdOpposableLogic}.
\end{proof}

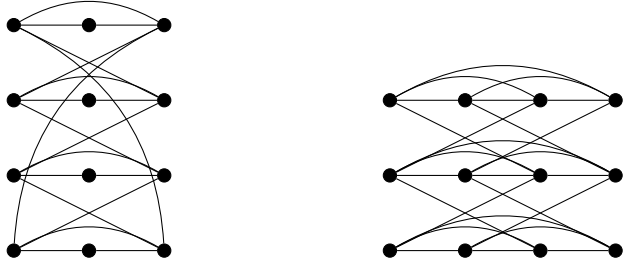
\begin{figure}
    \centering
    \begin{tikzpicture}
        \tikzstyle{vertex}=[circle, draw=black, fill=black, minimum
        size=5pt,inner sep=0pt]
        \foreach \x in {1,...,3}
        {
            \foreach \y in {1,...,4} 
            {
               \node[vertex]  (\x\y) at (\x,\y) {};
            }
        }
        \foreach \y in {1,...,4}
        {
            \draw (1\y) -- (2\y) -- (3\y);
            \path (3\y) edge [bend right] (1\y);
        }
        \foreach \y in {1,...,3}
        {
            \pgfmathtruncatemacro{\yo}{\y + 1}
            \draw (1\y) -- (3\yo);
            \draw (1\yo) -- (3\y);
        }
        \path (11) edge [bend left] (34);
        \path (14) edge [bend left] (31);

        \foreach \x in {1,...,4}
        {
            \foreach \y in {1,...,3} 
            {
               \node[vertex,xshift=5cm]  (\x\y) at (\x,\y) {};
            }
        }
        \foreach \y in {1,...,3}
        {
            \draw (1\y) -- (2\y) -- (3\y) -- (4\y);
            \path (4\y) edge [bend right] (1\y);
            \path (4\y) edge [bend right] (2\y);
            \path (3\y) edge [bend right] (1\y);
        }
        \foreach \y in {1,2}
        {
            \pgfmathtruncatemacro{\yo}{\y+1}
            \draw (1\y) -- (3\yo);
            \draw (3\y) -- (1\yo);
            \draw (2\y) -- (4\yo);
            \draw (4\y) -- (2\yo);
        }
    \end{tikzpicture}
    \caption{Opposable $C_4\ltimes P_3$ and non-opposable $P_3\ltimes C_4 \cong
    C_4\rtimes P_3$.}
    \label{fig:graph_product_opposability}
\end{figure}

To see that the conditions in Theorem \ref{thm:graphProdOpposableLogic}  are
necessary, we examine the graphs of \Cref{fig:graph_product_opposability}. Note
that $C_4\ltimes P_3$ meets the conditions of the theorem, and is thus
opposable. However, $P_3\ltimes C_4$ does not meet the conditions of the
theorem, and is not opposable, as the four vertices of degree five form a
clique and thus does not have an opposition. This example can be interpreted as
showing that some condition on $\varphi$ is necessary, as this graph could
instead be defined by swapping $G$ and $H$ (and the graph product accordingly)
to obtain the graph product $C_4\rtimes P_3$ which fails the $\varphi$
condition, or by interpreting it directly and noting  that it is necessary that
$G$ is opposable.

\section{Non-opposable Second-Player Win Graphs}
\label{sec: non-opposable 2nd player win graphs}

Kakihara \cite{kakihara:snort} conjectured that the second player wins
\textsc{snort} on a graph if and only if that graph is opposable. We show that
this conjecture is false by giving a number of counterexamples and
constructions. First, we show the example with the least number of vertices, as
determined by a simple computer search.

Consider the graph $P_3 \cup C_3$ illustrated in \cref{fig: P3 U C3}. We begin
by showing that $P_3 \cup C_3$ is not opposable. Since there are only two
pendants in $P_3 \cup C_3$, every automorphism of $P_3 \cup C_3$ either maps
the two pendants to themselves or to each other. Since automorphisms preserve
adjacency, every automorphism of $P_3 \cup C_3$ has the vertex of degree two in
the $P_3$ map to itself. Thus an opposition of $P_3 \cup C_3$ cannot exist
since an automorphism of $P_3 \cup C_3$ would be forced to have a fixed point.

Next we show that $P_3 \cup C_3$ is second-player win. Left has two options for
her first move. Either Left can play on the $P_3$ or the $C_3$. Suppose Left
plays on the $C_3$ for her first move. This reserves two vertices for Left that
Right cannot play on. In response, Right can play on the vertex of degree two
in the $P_3$. This reserves two vertices for Right that Left cannot play on.
Therefore both players have exactly two moves remaining with Left going first.
Thus Right wins. If instead Left played her first move on the $P_3$, she can
either play on a leaf or on the vertex of degree two. Regardless of which
vertex she plays on first, Right can play on any vertex in the $C_3$ which
preserves two vertices for himself that Left cannot play on. Therefore both
players have two moves remaining with Left going first and so Right wins.

\begin{figure}
    \centering
    \begin{tikzpicture}
        \tikzstyle{vertex}=[circle, draw=black, fill=black, minimum
        size=5pt,inner sep=0pt]

        \node[vertex] at (0,0) (v1) {};
        \node[vertex] at (0,1) (v2) {};
        \node[vertex] at (0,2) (v3) {};

        \node[vertex] at (3,0) (v4) {};
        \node[vertex] at (2,1) (v5) {};
        \node[vertex] at (3,2) (v6) {};

        \draw (v1)--(v2)--(v3);
        \draw (v4)--(v5)--(v6)--(v4);
    \end{tikzpicture}
    \caption{The smallest non-opposable graph that is second-player win.}
    \label{fig: P3 U C3}
\end{figure}
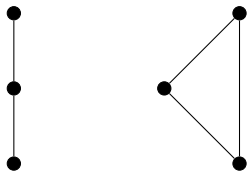

Next, we consider connectedness as a natural constraint on our graphs.
\cref{fig: smallest non-opposable 2nd player win connected graphs} illustrates
the five connected graphs on seven vertices that are non-opposable and
second-player win. By a computer search, among all connected, second-player
win, non-opposable graphs, the five graphs in \cref{fig: smallest non-opposable
2nd player win connected graphs} are all of the graphs with the fewest number
of vertices. Since all of the graphs in \cref{fig: smallest connected 2nd
player win graph (vertex and edge-wise)} have an odd number of vertices, they
are not opposable. In the following example, we show that one of the graphs is
second-player win. A similar proof can be done for the other four graphs.

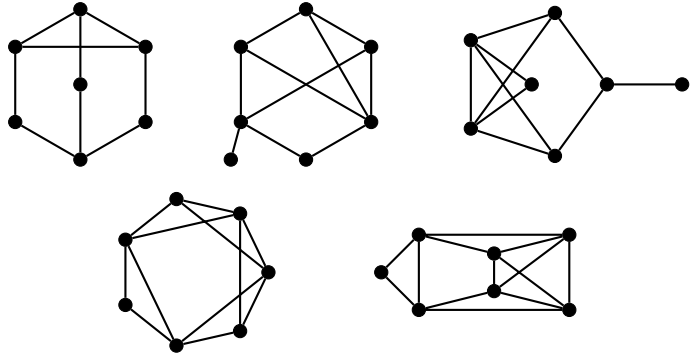
\begin{figure}
    \centering
    \begin{tikzpicture}
        \tikzstyle{vertex}=[circle, draw=black, fill=black, minimum
        size=5pt,inner sep=0pt]

        \foreach \x in {1,...,6} {
            \node[vertex] at ({90 + \x * 360/6}:1cm) (a\x) {};
        }
        \node[vertex] at (0,0) (a7) {};
        \draw[thick] (a1)--(a2)--(a3)--(a4)--(a5)--(a6)--(a1);
        \draw[thick] (a6)--(a7)--(a3);
        \draw[thick] (a5)--(a1);

        \foreach \x in {1,...,6} {
            \node[vertex, xshift=3cm] at ({90 + \x * 360/6}:1cm) (b\x) {};
        }
        \node[vertex, xshift=3cm] at (-1, -1) (b7) {};
        \draw[thick] (b2)--(b1)--(b6)--(b5)--(b4)--(b3)--(b2)--(b7);
        \draw[thick] (b2)--(b5);
        \draw[thick] (b1)--(b4)--(b6);

        \foreach \x in {1,...,5} {
            \node[vertex, xshift=6cm] at ({\x * 360/5}:1cm) (c\x) {};
        }
        \node[vertex, xshift=6cm] at (0,0) (c6) {};
        \node[vertex, xshift=6cm] at (2,0) (c7) {};
        \draw[thick] (c5)--(c4)--(c3)--(c2)--(c1)--(c5)--(c7);
        \draw[thick] (c2)--(c6)--(c3);
        \draw[thick] (c1)--(c3);
        \draw[thick] (c2)--(c4);

        \foreach \x in {1,...,7} {
            \node[vertex, xshift=1.5cm, yshift=-2.5cm] at ({\x * 360/7}:1cm)
            (d\x) {};
        }
        \draw[thick]
        (d6)--(d7)--(d1)--(d2)--(d3)--(d4)--(d5)--(d6)--(d1)--(d3)--(d5)--(d7)--(d2);

        \node[vertex, xshift=4.5cm, yshift=-3cm] at (0,0) (e1) {};
        \node[vertex, xshift=4.5cm, yshift=-3cm] at (0,1) (e2) {};
        \node[vertex, xshift=4.5cm, yshift=-3cm] at (2,1) (e3) {};
        \node[vertex, xshift=4.5cm, yshift=-3cm] at (2,0) (e4) {};
        \node[vertex, xshift=4.5cm, yshift=-3cm] at (1,0.25) (e5) {};
        \node[vertex, xshift=4.5cm, yshift=-3cm] at (1,0.75) (e6) {};
        \node[vertex, xshift=4.5cm, yshift=-3cm] at (-0.5,0.5) (e7) {};
        \draw[thick]
        (e1)--(e2)--(e3)--(e4)--(e1)--(e7)--(e2)--(e6)--(e3)--(e5)--(e6)--(e4)--(e5)--(e1);
    \end{tikzpicture}
    \caption{The five non-opposable, connected graphs of order seven that are
    second-player win.}
    \label{fig: smallest non-opposable 2nd player win connected graphs}
\end{figure}

Let $G$ be the graph in Figure \ref{fig: smallest connected 2nd player win
graph (vertex and edge-wise)}. Up to the symmetry of the graph, Left has five
options for her first move: $v_1$, $v_2$, $v_3$, $v_4$ and $v_7$.

Suppose on her first turn, Left colours $v_1$. In response, Right can colour
$v_4$. If Left colours $v_2$ on her second turn, then Right can colour $v_5$ to
win the game. If instead Left colours $v_6$ on her second turn, then Right can
colour $v_3$ to win the game. The analogous strategy works if Left begins with
$v_4$.

Suppose on her first turn, Left colours $v_2$. In response, Right can colour
$v_4$. If Left colours $v_1$ on her second turn, then Right can colour $v_5$ to
win the game. If instead Left colours $v_6$ on her second turn, then Right can
colour $v_7$ to win the game. Again, the analogous strategy works if Left
begins with $v_3$.

Suppose on her first turn, Left colours $v_7$. In response, Right can colour
$v_2$. If Left colours $v_4$ on her second turn, then Right can colour $v_6$ to
win the game. If instead Left colours $v_5$ on her second turn, then Right can
colour $v_3$ to win the game.

Therefore, regardless of which vertex Left colours on her first turn, Right can
still win the game. So, this graph is second-player win.

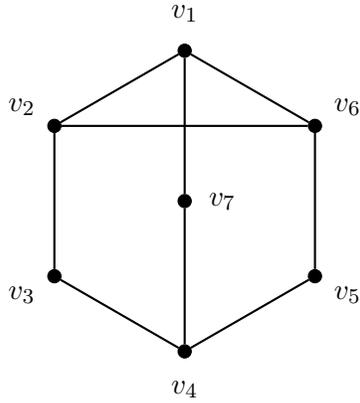
\begin{figure}
    \centering
    \begin{tikzpicture}
        \tikzstyle{vertex}=[circle, draw=black, fill=black, minimum
        size=5pt,inner sep=0pt]

        \foreach \x in {1,...,6} {
            \node[vertex] at ({90 + \x * 360/6}:2cm) (a\x) {};
        }
        \node[vertex] at (0,0) (a7) {};
        \draw[thick] (a1)--(a2)--(a3)--(a4)--(a5)--(a6)--(a1);
        \draw[thick] (a6)--(a7)--(a3);
        \draw[thick] (a5)--(a1);

        \foreach \x in {1,...,6} {
            \node at ({30 + \x * 360/6}:2.5cm) {$v_{\x}$};
        }
        \node at (0.5,0) {$v_7$};
    \end{tikzpicture}
    \caption{The graph with the fewest edges among all second-player win,
    connected graphs with seven vertices.}
    \label{fig: smallest connected 2nd player win graph (vertex and edge-wise)}
\end{figure}

\begin{figure}
    \centering
    \begin{tikzpicture}[scale=0.5]
        \tikzstyle{vertex}=[circle, draw=black, fill=black, minimum
        size=5pt,inner sep=0pt]

        \node[vertex] at (0,0) (v1) {};
        \node[vertex] at (1,0) (v2) {};
        \node[vertex] at (2,0) (v3) {};
        \node[vertex] at (3,0) (v4) {};
        \node[vertex] at (4,0) (v5) {};
        \node[vertex] at (5,0) (v6) {};
        \node[vertex] at (6,0) (v7) {};
        \node[vertex] at (2,1) (v8) {};
        \node[vertex] at (2,2) (v9) {};
        \node[vertex] at (2,-1) (v10) {};
        \node[vertex] at (2,-2) (v11) {};
        \node[vertex] at (4,1) (v12) {};

        \draw[thick] (v1)--(v2)--(v3)--(v4)--(v5)--(v6)--(v7);
        \draw[thick] (v9)--(v8)--(v3)--(v10)--(v11);
        \draw[thick] (v5)--(v12);

        \node[vertex, xshift=6cm] at (0,0) (u1) {};
        \node[vertex, xshift=6cm] at (1,0) (u2) {};
        \node[vertex, xshift=6cm] at (2,0) (u3) {};
        \node[vertex, xshift=6cm] at (3,0) (u4) {};
        \node[vertex, xshift=6cm] at (4,0) (u5) {};
        \node[vertex, xshift=6cm] at (5,0) (u6) {};
        \node[vertex, xshift=6cm] at (6,0) (u7) {};
        \node[vertex, xshift=6cm] at (3,3) (u8) {};
        \node[vertex, xshift=6cm] at (3,2) (u9) {};
        \node[vertex, xshift=6cm] at (3,1) (u10) {};
        \node[vertex, xshift=6cm] at (3,-1) (u11) {};
        \node[vertex, xshift=6cm] at (3,-2) (u12) {};
        \node[vertex, xshift=6cm] at (2,-3) (u13) {};
        \node[vertex, xshift=6cm] at (4,-3) (u14) {};
        \node[vertex, xshift=6cm] at (4,1) (u15) {};

        \draw[thick] (u1)--(u2)--(u3)--(u4)--(u5)--(u6)--(u7);
        \draw[thick] (u8)--(u9)--(u10)--(u4)--(u11)--(u12)--(u13);
        \draw[thick] (u14)--(u12);
        \draw[thick] (u15)--(u5);
    \end{tikzpicture}
    \caption{The unique smallest tree that is second-player win with an even
    number of vertices (left) and an odd number of vertices (right).}
    \label{fig: 2nd player winning 12 vertex tree}
\end{figure}
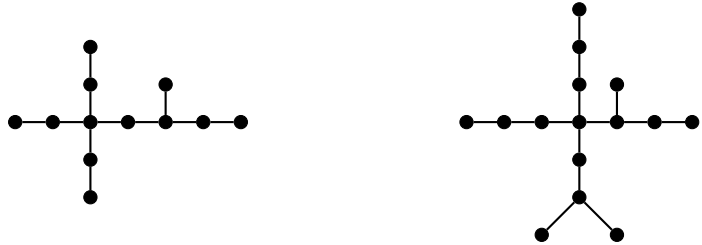

As a further natural constraint on our graphs, Figure \ref{fig: 2nd player
winning 12 vertex tree} illustrates two trees. As determined by a computer
search, out of all second-player win, non-opposable trees that have an even
(respectively odd) number of vertices, the tree with 12 (respectively 15)
vertices in Figure \ref{fig: 2nd player winning 12 vertex tree} is the smallest
in terms of number of vertices.

Next, we give examples of infinite families of graphs that are not opposable
and the second player wins. 

We construct an infinite family of graphs $\{G_{2n+1}\}_{n=2}^\infty$ where the
second player wins \textsc{snort} on each graph but none of the graphs are
opposable. For a fixed $n\geq 2$, begin with a $K_{2n+1}$. Take any two
vertices $u$ and $v$ of the $K_{2n+1}$ and subdivide the edge $uv$ by adding a
vertex $x$. Attach $2n-2$ pendants to $x$ to complete the construction of
$G_{2n+1}$. Figure \ref{fig: Will's infinite construction} illustrates $G_5$,
the graph in the family $\{G_{2n+1}\}_{n=2}^\infty$ that has the fewest number
of vertices.

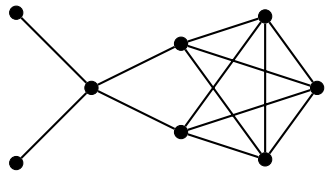
\begin{figure}
    \centering
    \begin{tikzpicture}
        \tikzstyle{vertex}=[circle, draw=black, fill=black, minimum
        size=5pt,inner sep=0pt]

        \foreach \x in {1,...,5} {
            \node[vertex] at ({\x * 360/5}:1cm) (a\x) {};
        }
        \node[vertex] at (-2,0) (x) {};
        \node[vertex] at (-3,1) (l1) {};
        \node[vertex] at (-3,-1) (l2) {};

        \draw[thick] (a1)--(a3)--(a5)--(a2)--(a4)--(a1);
        \draw[thick] (a2)--(a1)--(a5)--(a4)--(a3);
        \draw[thick] (l1)--(x)--(a2);
        \draw[thick] (l2)--(x);
        \draw[thick] (a3)--(x);
    \end{tikzpicture}
    \caption{A graph $G_5$ on eight vertices that is non-opposable and
    second-player win.}
    \label{fig: Will's infinite construction}
\end{figure}

To see that every graph in $\{G_{2n+1}\}_{n=2}^\infty$ is non-opposable, note
that for any $n\geq 2$, $G_{2n+1}$ contains exactly one vertex $x$ that is
adjacent to a pendant. Therefore every automorphism of $G_{2n+1}$ maps $x$ to
itself and so $G_{2n+1}$ is not opposable. Next, we show that the second player
always wins \textsc{snort} on any graph in $\{G_{2n+1}\}_{n=2}^\infty$.

Fix $n\geq 2$, with $u$, $v$ and $x$ as defined in the construction. Let $S$ be
the set of vertices in the $K_{2n+1}$ that are not $u$, $v$, or $x$. There are
four cases for Left's first move.

Suppose on her first move, Left colours a vertex in $S$. Right can respond by
colouring $x$. This results in the $2n-2$ pendants being accessible to Right
but not Left and $2n-2$ vertices in $S$ accessible to Left but not Right. Since
both players have $2n-2$ moves remaining with Left going first, Right wins.

Suppose on her first move, Left colours $u$ or $v$. In response, Right can
colour $v$ or $u$. Afterwards, the only vertices available to both players are
the $2n-2$ pendants. Since there are an even number of moves left in the game
with Left going first, Right wins.

Suppose on her first move, Left colours $x$. Right can respond by colouring any
of the vertices in $S$. This results in a similar position to the first case,
with each player having $2n-2$ moves remaining with Left going first. So Right
wins.

For the final case, suppose on her first move that Left colours a pendant.
Right can follow the same strategy as when Left began with colouring $x$, by
colouring any of the vertices in $S$. From this position, Left has at most
$2n-2$ moves remaining and Right has at least $2n-2$ moves remaining and Left
moves first. Therefore Right wins.

Next, we build a tool, the \emph{firework} of a graph, that will enable us to
construct infinitely many second-player win graphs from a given second-player
win graph. In particular, we use it to construct an infinite family of trees
that are second-player win.

\begin{definition}
    If $G$ is a graph that is second-player win, then we call a pair of
    vertices $(u,v)\in V(G)$ a \emph{response pair} if the game resulting from
    the first player playing on $u$, and the second player responding on $v$,
    is second-player win.
\end{definition}

Note that, if $(u,v)$ is a response pair, then so too is $(v,u)$. This is
because the game of \textsc{snort} on an (untinted) graph must be a symmetric
form; i.e.\ it is isomorphic to its conjugate ($G\cong\overline{G}$).

We note that response pairs may provide a way to generalize Kakihara's
conjecture that graphs are second-player win if and only if they have an
opposition. In essence, if the second player has a strategy which only looks at
the most recent move and not the full history, does this suffice to make the
graph opposable? 

\begin{problem}
    If there is a set of response pairs (second-player limited history winning
    strategies) $R=\{(u_i,v_i)\}$ such that all $u\in V(G)$ have some response
    pair $(u,v_i)\in R$, is it necessary and sufficient that $G$ is opposable?
    Are response pairs the correct definition to capture that the second player
    only looks at the previous move? 
\end{problem}

Observing the opposable graph $4K_1$ may enable a set of response pairs in a
$4$-cycle, we note that an opposition may not be immediately extracted, yet the
graph remains opposable.

\begin{definition}
    \label{def:firework}
    If $G$ is a graph that is second-player win, and $(u,v)$ is a response pair
    such that $\dist(u,v)\geq3$, then we construct the $n$\emph{-firework} of
    $G$ corresponding to the pair $(u,v)$ by adding $n$ leaves to $u$ and $n$
    leaves to $v$.
\end{definition}

Using response pairs, we may also construct an infinite family of trees where
the second player wins. In \cref{def:firework}, one might like to refer to the
vertices $u$ and $v$ as the \emph{shells}, and the leaves attached to $u$ and
$v$ as the \emph{confetti} (a single leaf in particular would be a
\emph{confetto}).

\begin{lemma}
    \label{lem: bad-moves}
    If $G$ is a graph such that the neighbours of $u$ are a subset of the
    neighbours of $v$, then, when playing \textsc{snort} on $G$, it is always
    at least as good to play on $v$ as it is to play on $u$.
\end{lemma}

\begin{proof}
    Suppose, without loss of generality, that at some moment in the game Left
    has the opportunity to play on either $u$ or $v$. We must show that, if she
    has a winning strategy playing on $u$, then she also has a winning strategy
    playing on $v$.

    Notice that playing on $u$ removes $u$ and tints its open neighbourhood
    blue, and similarly for playing on $v$. However, every neighbour of $u$ is
    also a neighbour of $v$, and so the move on $v$ is at least as advantageous
    for Left as the move on $u$; the only possible difference between the graph
    obtained by moving on $v$ and the graph obtained by moving on $u$ is that
    the former may have more vertices tinted blue, which can only help Left.
\end{proof}

\begin{theorem}
    \label{thm: firework}
    If the second player wins \textsc{snort} on a graph $G$ with a response
    pair $(u,v)$ of distance at least 3, then the second player wins
    \textsc{snort} on the $n$-firework of $G$ for all $n\geq|V(G)|$.
\end{theorem}

\begin{proof}
    Consider playing the $n$-firework $F$ of $G$ constructed from the response
    pair $(u,v)$. Since this is a symmetric form, we need only show that Right
    wins playing second.

    If Left plays on $u$, then Right can win by playing on $v$ since $(u,v)$ is
    a response pair. Similarly, if Left plays on $v$, then Right can win by
    playing on $u$. So suppose now that Left plays on neither $u$ nor $v$.
    Since $\dist_G(u,v)\geq3$, from the construction of a firework we have that
    $\dist_F(u,v)\geq3$. So, we may assume without loss of generality that
    Right may respond on $v$.

    By \cref{lem: bad-moves}, we may assume that Left responds either on $u$ or
    on some other vertex of $G$; in particular, Left does not play on a
    confetto. If Left plays on $u$, then Right can resume the second player
    winning strategy, since the order of Left's moves does not matter.
    Otherwise, Right can play either on $u$ or some confetto of $u$. In either
    case, what results is necessarily the disjunctive sum of a subposition $G'$
    of the original game $G$, a sum of copies of $*$ (from the confetti, if
    both players play adjacent to $u$), and the negative integer $-n$ (from the
    leaves of $v$).
    
    Since $|V(G)|\leq n$, the subposition $G'$ of $G$ has value $G'\leq n-1$
    (some vertices of $G$ have been removed). Thus, Right wins the resulting
    graph.
\end{proof}

In many cases, \cref{thm: firework} could be improved by allowing $n$ to in
fact be significantly smaller than $|V(G)|$, but we have no need for such
increased power here; we already have all we need to find our second-player win
trees.

\begin{corollary}
    There exists an infinite number of trees on which the second player wins
    \textsc{snort}.
\end{corollary}

\begin{proof}
    Consider the tree in \cref{fig:response-pair-tree}. It is a straightforward
    calculation to show that it is second-player win and has a response pair of
    distance three. Thus, by \cref{thm: firework}, we have the result; notice
    that the firework construction only adds leaves, and so every firework of a
    tree is itself a tree.
\end{proof}

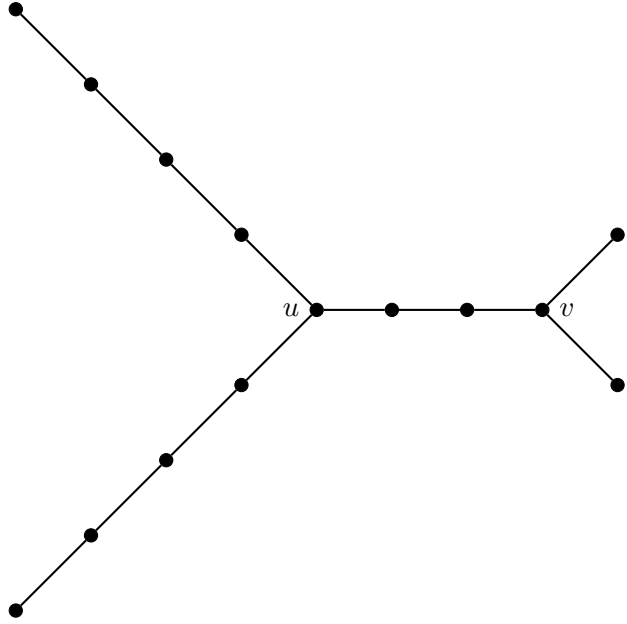
\begin{figure}
    \centering
    \begin{tikzpicture}
        \tikzstyle{vertex}=[circle, draw=black, fill=black, minimum
        size=5pt,inner sep=0pt]

        \node[vertex] at (-4,4) (v0) {};
        \node[vertex] at (-3,3) (v1) {};
        \node[vertex] at (-2,2) (v2) {};
        \node[vertex] at (-1,1) (v3) {};
        \node[vertex, label=left: $u$] at (0,0) (v4) {};
        \node[vertex] at (1,0) (v5) {};
        \node[vertex] at (2,0) (v6) {};
        \node[vertex, label=right: $v$] at (3,0) (v7) {};
        \node[vertex] at (4,1) (v8) {};
        \node[vertex] at (4,-1) (v9) {};
        \node[vertex] at (-1,-1) (v10) {};
        \node[vertex] at (-2,-2) (v11) {};
        \node[vertex] at (-3,-3) (v12) {};
        \node[vertex] at (-4,-4) (v13) {};

        \draw[thick] (v0)--(v1)--(v2)--(v3)--(v4)--(v5)--(v6)--(v7)--(v8);
        \draw[thick] (v7)--(v9);
        \draw[thick] (v4)--(v10)--(v11)--(v12)--(v13);
    \end{tikzpicture}
    \caption{A second-player win tree with a response pair $(u,v)$ of distance
    3.}
    \label{fig:response-pair-tree}
\end{figure}

We can use a similar trick to create untinted, second-player win graphs from
tinted, second-player win graphs. The trick is to add a new vertex and make it
adjacent to all the blue-tinted vertices, and add another vertex and make it
adjacent to all the red-tinted vertices. Then, as in the firework construction,
add many leaves to each of these two new vertices. If one adds a sufficiently
large number of leaves, then the resulting graph must be second-player win. 

\section{Almost Opposable Graphs}
\label{sec: almost opposable graphs}

In this section, we define a graph property analogous to opposable where the
player moving first is able to use a ``mirroring'' strategy to win. 

\begin{definition}
    \label{def: almost opposable}
    Let $G$ be a graph. We say that a set $S\subseteq V(G)$ is
    \emph{admissible} if $u\in S\subseteq N[u]$ for some $u\in V(G)$. If there
    exists an admissible set $S$ such that $G-S$ is opposable then we say that
    $G$ is \emph{almost opposable}.
\end{definition}

Definition \ref{def: almost opposable} is a generalization of one of the
automorphisms studied by Arroyo \cite{arroyo:dawsons}. In their thesis
\cite{arroyo:dawsons}, Arroyo considered automorphisms of order two on graphs
with an odd number of vertices such that these automorphisms had exactly one
fixed point and, besides the fixed point, the automorphism mapped vertices
outside of their closed neighbourhoods. Graphs with such automorphisms satisfy
our definition of almost opposable because deleting the vertex that is a fixed
point of the automorphism yields a graph that has an opposition. 

Now we show that the first player wins on almost opposable graphs. 

\begin{lemma}
    \label{lem: a.o. implies 1st player win}
    If $G$ is almost opposable, then the first player wins \textsc{snort} on
    $G$.
\end{lemma}

\begin{proof}
    Let $v$ be a vertex in $G$ such that for some $v\in S\subseteq N[v]$, $G-S$
    is opposable. Fix $S$ to be any such subset of $N[v]$. Let $f$ be an
    opposition of $G-S$. The first player can win by implementing the following
    strategy. On her first move, the first player colours $v$. By doing this,
    every vertex that the second player colours must be in $G-N[v]$.
    Consequently, every vertex the second player colours is in the subgraph
    $G-S$. Since $G-S$ is opposable, whenever the second player colours $u$ in
    $G-S$, the first player can respond by colouring $f(u)$. By the same
    inductive argument as in the proof of Lemma \ref{lem: opposable implies
    second player wins snort}, the first player can continue this strategy
    until the second player runs out of moves. 
\end{proof}

The converse of Lemma \ref{lem: a.o. implies 1st player win} does not hold. By
considering brooms, we can obtain infinitely many trees where the first player
wins, but the graphs are not almost opposable.

A broom graph $B(n,m)$ is a path on $n$ vertices with $m$ leaves attached to
one end. In particular, the longest path in $B(n,m)$ has $n+1$ vertices. We
claim that brooms are not almost opposable when $n\geq5$ and $m\geq3$ and when
$n\geq8$ and $m=2$.

If $n\geq5$ and $m\geq3$, then the vertex $v$ adjacent to the $m$ leaves has
degree at least four. Any vertex deletions that do not include $v$ can remove
at most one vertex adjacent to $v$. So the resulting graph would have a single
vertex of degree at least three and therefore no opposition. Deleting an
admissible set that includes $v$ must leave a path on at least two vertices,
which is not opposable since any automorphism will have either a fixed point or
adjacent vertices mapped to each other at the centre of the path.

If $n=5,\,6$ or $7$ and $m=2$, then it is possible to remove one, two, or three
path vertices respectively to yield a $P_3\cup P_3$, which allows an
opposition. If $n\geq8$ and $m=2$, then the broom is not almost opposable since
the deletion of any admissible set will leave either a single vertex of degree
greater than two, the disjoint union of two paths of different lengths, a
single path, or a path plus two isolated vertices. Observing none of these are
opposable, we conclude $B(n,2)$ is not almost opposable for $n\geq 8$.

To show that the first player wins on infinitely many brooms, we can use the
\emph{temperature} of \textsc{snort} played on a path. For the definition of
the temperature of a combinatorial game we direct the reader to
\cite{huntemann.nowakowski.ea:bounding}. Suppose the first player makes their
initial move on the end of the path adjacent to the $m$ leaves. This leaves a
path on $n-1$ vertices to play on, with one end only available to the first
player. Since the temperature of paths in \textsc{snort} is at most $7.5$
\cite{huntemann.nowakowski.ea:bounding}, with the second player having the next
move they can reserve no more than $8$ vertices in optimal play. Therefore if
$m\geq9$ and $n\in \mathbb{Z}^+$, then the first player wins on $B(n,m)$.

Unlike opposability, graph products do not behave quite as well with almost
opposability. A stronger condition can instill better behaviour.

\begin{definition}
    Let $G$ be an almost opposable graph. If there exists an automorphism of
    order two $\alpha: V(G) \to V(G)$, an admissible set $S$, and an opposition
    $f$ of $G-S$ such that for all $x\in V(G-S)$, $\alpha(x) = f(x)$, then we
    say that $f$ is a \emph{compatible} opposition. If $G-S$ has a compatible
    opposition for some admissible set $S$ then we say that $G$ is
    \emph{compatibly} almost opposable. 
\end{definition}

\begin{theorem}
    \label{thm: a.o. graphs with automorphism strong product is a.o.}
    If $G$ and $H$ are compatibly almost opposable graphs then $G\boxtimes H$
    is compatibly almost opposable. 
\end{theorem}

\begin{proof}
    Let $u\in V(G)$ and $u\in S_G\subseteq N_G[u]$ such that $G-S_G$ is
    opposable. Let $g$ be a compatible opposition of $G-S_G$. Let $v\in V(H)$
    and $v\in S_H \subseteq N_H[v]$ such that $H-S_H$ is opposable. Let $h$ be
    a compatible opposition of $H-S_H$. Let $g^\prime :V(G) \to V(G)$ and
    $h^\prime : V(H) \to V(H)$ be automorphisms of order two such that for all
    $x \in V(G) \backslash S_G$ and $y\in V(H) \backslash S_H$, $g^\prime(x) =
    g(x)$ and $h^\prime(y) = h(y)$. Let $S = \{(x,y) \in V(G\boxtimes H)\mid
    x\in S_G, \ y\in S_H\}$. Note that $(u,v) \in S\subseteq N_{G\boxtimes
    H}[(u,v)]$. Define $f: V(G\boxtimes H - S) \to V(G\boxtimes H - S)$ by
    $f((x,y)) = (g^\prime(x), h^\prime(y))$. We claim that $f$ is an opposition
    of $G\boxtimes H - S$.

    Let $(x_1, y_1), (x_2, y_2) \in V(G\boxtimes H - S)$ be distinct. Since
    $g^\prime$ and $h^\prime$ are bijections, $f$ is also a bijection. Suppose
    $(x_1, y_1) \sim_{G\boxtimes H - S} (x_2, y_2)$. If $x_1 = x_2$ then
    $g^\prime (x_1) = g^\prime (x_2)$. If $x_1 \sim_G x_2$ then $g^\prime (x_1)
    \sim_G g^\prime(x_2)$. Similar statements can be made for $y_1$, $y_2$ and
    $h^\prime$. Since $g^\prime$ never maps a vertex outside of $S_G$ to a
    vertex in $S_G$ and $h^\prime$ never maps a vertex outside of $S_H$ to a
    vertex in $S_H$, $f((x_1, y_1)) \sim_{G\boxtimes H - S} f((x_2, y_2))$.
    Therefore, $f$ is an automorphism of $G\boxtimes H - S$.

    Let $(x,y) \in V(G\boxtimes H - S)$. Since $g^\prime$ and $h^\prime$ are
    automorphisms of order two, $f(f((x,y))) = (x,y)$. It remains to show that
    $(x,y) \nsim_{G\boxtimes H - S} f((x,y))$. By the way $S$ is defined,
    either $x\notin S_G$ or $y\notin S_H$. If $x\notin S_G$ then since $g$ is
    an opposition, $g(x) = g^\prime (x) \nsim_G x$. Thus, $(x,y)
    \nsim_{G\boxtimes H - S} (g^\prime(x), h^\prime(y)) = f((x,y))$. If instead
    $x\in S_G$ then $y\notin S_H$. By a similar argument, $h^\prime (y) \nsim_H
    y$ and so $(x,y) \nsim_{G\boxtimes H - S} (g^\prime(x), h^\prime(y)) =
    f((x,y))$. Therefore, $f$ is an opposition. 

    Let $f^\prime : G\boxtimes H \to G\boxtimes H$ be defined by
    $f^\prime((x,y)) = (g^\prime(x), h^\prime(y))$. Using a similar argument as
    with $f$, we have that $f^\prime$ is an automorphism of order two.
    Furthermore, for any $(x,y) \in V(G\boxtimes H - S)$, $f^\prime((x,y)) =
    (g^\prime(x), h^\prime(y)) = f((x,y))$. This proves the theorem. 
\end{proof}

By Theorem \ref{thm: a.o. graphs with automorphism strong product is a.o.}, if
$G$ and $H$ satisfy the conditions of Theorem \ref{thm: a.o. graphs with
automorphism strong product is a.o.} then so does their strong product
$G\boxtimes H$. As a consequence of this, given any finite number of graphs
$G_1, \dots, G_k$ that satisfy the conditions of Theorem \ref{thm: a.o. graphs
with automorphism strong product is a.o.}, the strong product
$\boxtimes^k_{i=1} G_i$ also satisfies the conditions of Theorem \ref{thm: a.o.
graphs with automorphism strong product is a.o.}. One example of this is the
$k$-dimensional strong grid. Since $k$-dimensional strong grids will be
mentioned in Section \ref{sec: peaceable queens game}, here we give the proof
that all $k$-dimensional strong grids are almost opposable. 

\begin{corollary}
    \label{cor: strong grid is a.o.}
    If $k, n_i\in \mathbb{Z}^+$ for each $1\leq i\leq k$, then
    $\boxtimes_{i=1}^k P_{n_i}$ is almost opposable.  
\end{corollary}

\begin{proof}
    Let $n\in \mathbb{Z}^+$ and let $v_0, \dots, v_{n-1}$ be the vertices of
    $P_n$. Consider the function $f: V(P_n) \to V(P_n)$ defined by $f(v_i) =
    v_{n-i-1}$. Note that $f$ is an automorphism of order two. If $n$ is odd
    then $f$, when restricted to the domain $V(P_n - v_{\frac{n-1}{2}})$, is an
    opposition of $P_n - v_{\frac{n-1}{2}}$. If $n$ is even then $f$, when
    restricted to the domain $V(P_n-\{v_{\frac{n}{2}-1},v_{\frac{n}{2}}\})$, is
    an opposition of $P_n-\{v_{\frac{n}{2}-1},v_{\frac{n}{2}}\}$.

    Therefore, $P_n$ is compatibly almost opposable for any $n\in
    \mathbb{Z}^+$. Thus, by Theorem \ref{thm: a.o. graphs with automorphism
    strong product is a.o.}, $\boxtimes^k_{i=1} P_{n_i}$ where $k, n_1, \dots,
    n_k \in \mathbb{Z}^+$ is almost opposable. 
\end{proof}

The condition in Theorem \ref{thm: a.o. graphs with automorphism strong product
is a.o.} of being compatibly almost opposable is necessary to guarantee that
the strong product is almost opposable. The following counterexample
demonstrates this. 

Let $G$ be the graph in Figure \ref{fig: a.o. auto. strong prod.
counterexample}, with vertices labelled in the same way. Note that $G$ is
almost opposable since $G - \{c,d,g\} \cong P_2 \cup P_2$ is opposable and
$\{c,d,g\} \subseteq N[d]$. There are two oppositions of $G - \{c,d,g\}$: the
automorphism of order two that maps $a$ to $f$ and $b$ to $e$, and the
automorphism of order two that maps $a$ to $e$ and $b$ to $f$. We claim that
$\{c,d,g\}$ is the only admissible sets of vertices such that deleting them
from $G$ yields an opposable graph. Since $G$ has an odd number of vertices,
the number of vertices we delete must be odd for an opposition to be possible.
Deleting any single vertex from $G$ will not yield an opposable graph. The only
admissible sets of vertices of size three are $\{a,b,c\}$, $\{b,c,d\}$,
$\{c,d,e\}$, $\{c,d,g\}$, $\{d,e,g\}$ and $\{d,e,f\}$. Note that $G-\{a,b,c\}
\cong P_4$, $G-\{b,c,d\} \cong 2K_1 \cup P_2$, $G-\{c,d,e\} \cong 2K_1 \cup
P_2$, $G-\{d,e,g\} \cong P_3 \cup K_1$ and $G-\{d,e,f\} \cong P_3 \cup K_1$
which are all not opposable. There are no admissible sets of size five or
larger and so $\{c,d,g\}$ is the only admissible set. 

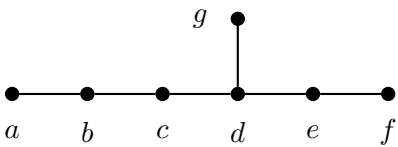
\begin{figure}
    \centering
    \begin{tikzpicture}
        \tikzstyle{vertex}=[circle, draw=black, fill=black, minimum
        size=5pt,inner sep=0pt]

        \node[vertex] at (0,0) (a) {};
        \node[vertex] at (1,0) (b) {};
        \node[vertex] at (2,0) (c) {};
        \node[vertex] at (3,0) (d) {};
        \node[vertex] at (4,0) (e) {};
        \node[vertex] at (5,0) (f) {};
        \node[vertex] at (3,1) (g) {};

        \draw[thick] (a)--(b)--(c)--(d)--(e)--(f);
        \draw[thick] (d)--(g);

        \node[yshift=-0.5cm] at (0,0) {$a$};
        \node[yshift=-0.5cm] at (1,0) {$b$};
        \node[yshift=-0.5cm] at (2,0) {$c$};
        \node[yshift=-0.5cm] at (3,0) {$d$};
        \node[yshift=-0.5cm] at (4,0) (e) {$e$};
        \node[yshift=-0.5cm] at (5,0) (f) {$f$};
        \node[xshift=-0.5cm] at (3,1) (g) {$g$};
    \end{tikzpicture}
    \caption{An almost opposable graph.}
    \label{fig: a.o. auto. strong prod. counterexample}
\end{figure}

First we show that $G$ does not have a nontrivial automorphism. Let $\alpha$ be
an automorphism of $G$. The three leaves, $a$, $f$ and $g$, would be forced to
be mapped to each other by $\alpha$. Since $d$ is the unique vertex of degree
three, $\alpha(d) = d$. However, $\alpha$ must have that $d(x,d) = d(\alpha(x),
d)$ for each $x\in \{a,f,g\}$. The only way this is possible is if each of the
leaves are mapped to themselves. It follows that $\alpha$ maps every vertex to
itself. Therefore $G$ does not have an automorphism that maps $a$ to $f$ and
$b$ to $e$ nor an automorphism that maps $a$ to $e$ and $b$ to $f$. Thus $G$ is
not compatibly almost opposable. 

Consider the graph $G\boxtimes P_4$. We claim that $G\boxtimes P_4$ is not
almost opposable. Let $v_1, v_2, v_3, v_4$ be the vertices of $P_4$. Suppose
for a contradiction that $G\boxtimes P_4$ is almost opposable. Let $S$ be an
admissible set of vertices such that $G\boxtimes P_4 - S$ is opposable. Since
$G$ has no nontrivial automorphisms and since $S$ cannot contain vertices from
both $G.v_1$ and $G.v_4$, an opposition of $G\boxtimes P_4 - S$ must map every
vertex $(x,v_1)\in V(G.v_1)$ to $(x,v_4)\in V(G.v_4)$. Consequently, every
vertex $(x,v_2)\in V(G.v_2)$ must map to $(x,v_3)\in V(G.v_3)$. However,
$(x,v_2) \sim (x, v_3)$. Since $S$ cannot contain all vertices in $G.v_2$ and
$G.v_3$, an opposition of $G\boxtimes P_4 - S$ cannot exist.

While the conditions in Theorem \ref{thm: a.o. graphs with automorphism strong
product is a.o.} are necessary to guarantee that the strong product of two
graphs is almost opposable, the converse of Theorem \ref{thm: a.o. graphs with
automorphism strong product is a.o.} does not hold. 

Let $v_0$, $v_1$ and $v_2$ be the vertices of $P_3$ with $v_0$ and $v_2$ being
the leaves. Let $G$ be the graph in Figure \ref{fig: a.o. auto. strong prod.
counterexample}. Recall that $G$ does not have any nontrivial automorphisms. We
claim that the graph $G\boxtimes P_3$, illustrated in Figure \ref{fig:
counterexample to converse of a.o. strong prod thm} is almost opposable. Let $S
= \{ (c,x) | x\in \{0,1,2\} \} \cup \{ (d,x) | x\in \{0,1,2\} \} \cup \{ (g,x)
| x\in \{0,1,2\} \}$. Note that $S\subseteq N_{G\boxtimes P_3}[(d,v_1)]$. The
graph $G\boxtimes P_3 - S$ is isomorphic to $(P_2 \boxtimes P_3) \cup (P_2
\boxtimes P_3)$ which is opposable. Therefore $G\boxtimes P_3$ is almost
opposable, even though $G$ is not almost opposable. 

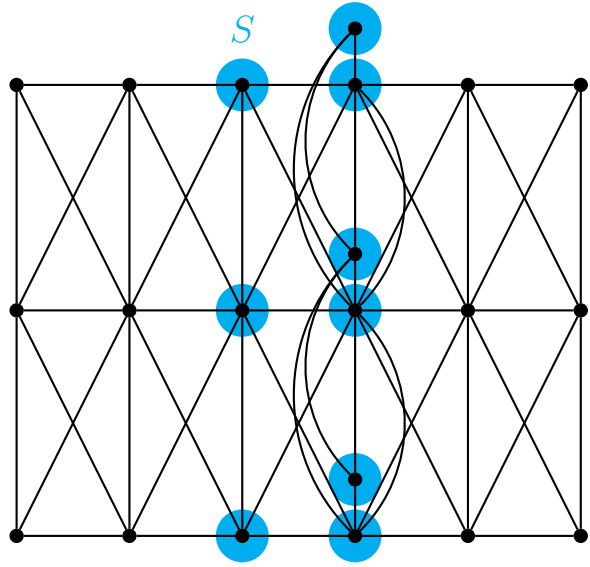
\begin{figure}
    \centering
    \begin{tikzpicture}[scale=0.75]
        \tikzstyle{vertex}=[circle, draw=black, fill=black, minimum
        size=5pt,inner sep=0pt]

        \node at (4, 1) {\Large{\textcolor{cyan}{$S$}}};

        \foreach \x in {0,...,2}
        {

        \node[circle, fill=cyan, minimum size=20pt, yshift=-3*\x cm] at (4,0) {};
        \node[circle, fill=cyan, minimum size=20pt, yshift=-3*\x cm] at (6,0) {};
        \node[circle, fill=cyan, minimum size=20pt, yshift=-3*\x cm] at (6,1) {};

        \node[vertex, yshift=-3*\x cm] at (0,0) (a\x) {};
        \node[vertex, yshift=-3*\x cm] at (2,0) (b\x) {};
        \node[vertex, yshift=-3*\x cm] at (4,0) (c\x) {};
        \node[vertex, yshift=-3*\x cm] at (6,0) (d\x) {};
        \node[vertex, yshift=-3*\x cm] at (8,0) (e\x) {};
        \node[vertex, yshift=-3*\x cm] at (10,0) (f\x) {};
        \node[vertex, yshift=-3*\x cm] at (6,1) (g\x) {};

        \draw[thick] (a\x)--(b\x)--(c\x)--(d\x)--(e\x)--(f\x);
        \draw[thick] (d\x)--(g\x);

        }

        \foreach \x in {a,b,c,e,f}
            \draw[thick] (\x 0)--(\x 1)--(\x 2);

        \draw[thick] (a0)--(b1)--(a2);
        \draw[thick] (b0)--(a1)--(b2)--(c1)--(b0);
        \draw[thick] (c0)--(b1)--(c2)--(d1)--(c0);
        \draw[thick] (d0)--(c1)--(d2)--(e1)--(d0);
        \draw[thick] (e0)--(d1)--(e2)--(f1)--(e0);
        \draw[thick] (f0)--(e1)--(f2);
        \draw[thick] (d0)--(g1);
        \draw[thick] (d1)--(g2);

        \path[thick] (g0) edge [bend right=45] (g1);
        \path[thick] (g1) edge [bend right=45] (g2);
        \path[thick] (d0) edge [bend left=45] (d1);
        \path[thick] (d1) edge [bend left=45] (d2);
        \path[thick] (g0) edge [bend right=45] (d1);
        \path[thick] (g1) edge [bend right=45] (d2);
    \end{tikzpicture}
    \caption{A counterexample for the converse of \cref{thm: a.o. graphs with
    automorphism strong product is a.o.}.}
    \label{fig: counterexample to converse of a.o. strong prod thm}
\end{figure}

The property of $P_3$ that allows this strong product to be almost opposable is
the property that $P_3$ contains a vertex that is adjacent to every other
vertex in $P_3$. We will call such a vertex a \emph{universal} vertex. We can
generalize this result to any almost opposable graph $G$ and any graph with a
universal vertex $H$. 

\begin{theorem}
    \label{thm: a.o. strong product uni. vertex is a.o.}
    If $G$ is almost opposable and $H$ contains a universal vertex, then
    $G\boxtimes H$ is almost opposable. 
\end{theorem}

\begin{proof}
    Let $v$ be a universal vertex of $H$. Let $u\in V(G)$ and $u\in S_G
    \subseteq N_G[u]$ such that $G-S_G$ is opposable. Let $g$ be an opposition
    of $G-S_G$. Let $S = \{ (x,y)\in V(G\boxtimes H) \mid x\in S_G, y\in
    V(H)\}$. Note that $(u,v) \in S\subseteq N_{G\boxtimes H}[(u,v)]$. Define
    $f: V(G\boxtimes H - S) \to V(G\boxtimes H - S)$ by $f((x,y)) = (g(x), y)$.
    We claim that $f$ is an opposition of $G\boxtimes H - S$. 

    Let $(x_1, y_1), (x_2, y_2) \in V(G\boxtimes H - S)$ be distinct. Since $g$
    is a bijection, $f$ is also a bijection. Suppose $(x_1, y_1)
    \sim_{G\boxtimes H - S_G} (x_2, y_2)$. If $x_1 = x_2$ then $g(x_1) =
    g(x_2)$. If $x_1 \sim_{G - S_G} x_2$ then $g(x_1) \sim _{G - S_G} g(x_2)$.
    Since $g$ is an automorphism of $G- S_G$, $g$ never maps a vertex outside
    of $S_G$ to a vertex in $S_G$. So $f((x_1,y_1)) = (g(x_1), y_1)
    \sim_{G\boxtimes H - S} (g(x_2), y_2) = f((x_2, y_2))$ and $f$ is an
    automorphism of $G\boxtimes H - S$.

    Let $(x,y) \in V(G\boxtimes H - S)$. Note that $f(f((x,y))) = (x,y)$ and so
    $f$ is of order two. Since $g$ is an opposition of $G-S_G$, $g(x) \nsim_G
    x$. Thus $(x,y) \nsim_{G\boxtimes H - S} (g(x), y) = f((x,y))$. 

    Therefore $f$ is an opposition of $G\boxtimes H - S$ and so $G\boxtimes H$
    is almost opposable. 
\end{proof}

Using similar criteria for two graphs $G$ and $H$ as in Theorem \ref{thm: a.o.
graphs with automorphism strong product is a.o.}, we can construct larger
almost opposable graphs using the Cartesian product.

\begin{theorem}
    \label{thm: a.o. graphs with automorphism Cartesian product is a.o.}
    If $G$ and $H$ are compatibly almost opposable, and there exists a $u\in
    V(G)$ such that $G-u$ is opposable, then $G\square H$ is compatibly almost
    opposable.
\end{theorem}

\begin{proof}
    Let $u\in V(G)$ such that $G-u$ is opposable. Let $g$ be a compatible
    opposition of $G-u$. Let $v\in V(H)$ and $v\in S_H \subseteq N_H[v]$ such
    that $H-S_H$ is opposable. Let $h$ be a compatible opposition of $H-S_H$.
    Let $g^\prime :V(G) \to V(G)$ and $h^\prime : V(H) \to V(H)$ be
    automorphisms of order two such that for all $x \in V(G) \backslash S_G$
    and $y\in V(H) \backslash S_H$, $g^\prime(x) = g(x)$ and $h^\prime(y) =
    h(y)$ Let $S = \{(x,y) \in V(G\square H) | x=u, y\in S_H\}$. Define $f:
    V(G\square H - S) \to V(G\square H - S)$ by $f((x,y)) = (g^\prime(x),
    h^\prime(y))$. We claim that $f$ is an opposition of $G\square H - S$. 

    By using the same argument as in the proof of Theorem \ref{thm: a.o. graphs
    with automorphism strong product is a.o.}, we get that $f$ is well-defined,
    injective and surjective. Let $(x_1, y_1), (x_2, y_2) \in V(G\square H -
    S)$. If $(x_1, y_1) \sim_{G\square H - S} (x_2, y_2)$ then either $x_1=x_2$
    and $y_1y_2\in E(H)$ or $x_1x_2\in E(G)$ and $y_1=y_2$. If $x_1=x_2$ then
    since $g^\prime$ and $h^\prime$ are automorphisms, $f((x_1,y_1)) =
    (g^\prime(x_1), h^\prime(y_1)) = (x_2, h^\prime(y_1)) \sim_{G\square H - S}
    (x_2, h^\prime(y_2)) = (g^\prime(x_2), h^\prime(y_2)) = f((x_2,y_2))$. Thus
    $f$ is an automorphism and $f$ is of order two. If $(x_1, y_1)
    \nsim_{G\square H - S} (x_2, y_2)$ then either $x_2 \notin N_{G}[x_1]$ or
    $y_2\notin N_H[y_1]$. Without loss of generality, suppose $x_2 \notin
    N_G[x_1]$. If $x_1\neq u$ then $g^\prime(x_2) \notin N_{G}[g^\prime(x_1)]$
    and so $f((x_1, y_1)) \nsim_{G\square H - S} f((x_2, y_2))$. If $x_1 = u$
    then $y_1 \notin S_H$. Thus $h^\prime(y_1) \nsim_{H-S_H} h^\prime(y_2)$ and
    so $f((x_1, y_1)) \nsim_{G\square H - S} f((x_2, y_2))$. Therefore $f$ is
    an opposition of $G\square H - S$.

    Define $f^\prime : V(G\square H) \to V(G\square H)$ by $f((x,y)) =
    (g^\prime(x), h^\prime(y))$. Using a similar argument as with $f$, we have
    that $f$ is an automorphism of order two. Furthermore, for any $(x,y) \in
    V(G\square H - S)$ we have that $f^\prime((x,y)) = f((x,y))$. This proves
    the theorem. 
\end{proof}

Independently of Kakihara and Arroyo, Uiterwijk \cite{uiterwijk:solving}
characterized the outcomes of \textsc{snort} on $P_n \square P_m$ by using the
opposability of $P_n \square P_m$ when both $n$ and $m$ are even and the almost
opposability of $P_n \square P_m$ when either $n$ and $m$ are odd. By using
\cref{cor: graph products preserve opposable} and \cref{thm: a.o. graphs with
automorphism Cartesian product is a.o.}, we can generalize Uiterwijk's result
to $k$-dimensional Cartesian grids.

\begin{corollary}
    \label{cor: k-dim Cartesian grids}
    If $k\in \mathbb{Z}^+$ and $n_1,\dots, n_k \in \mathbb{Z}^+$, then the
    graph $\square^k_{i=1} P_{n_i}$ is almost opposable if and only if there is
    at most one $n_i$ that is even; otherwise, $\square^k_{i=1} P_{n_i}$ is
    opposable.
\end{corollary}

\begin{proof}
    Without loss of generality, suppose $n_1$ and $n_2$ are even. Note that
    $\square^k_{i=1} P_{n_i}$ is isomorphic to $(P_{n_1} \square P_{n_2})
    \square (\square_{\ell = 3}^k P_{n_\ell})$.

    We claim that $P_{n_1} \square P_{n_2}$ is opposable. Let $x_1, \dots,
    x_{n_1}$ be the vertices of $P_{n_1}$ such that $x_i \sim x_{i+1}$ for each
    $1\leq i\leq n_1 - 1$. Let $y_1, \dots, y_{n_2}$ similarly denote the
    vertices of $P_{n_2}$. The function $f: V(P_{n_1}) \to V(P_{n_2})$ defined
    by $f((x_i, y_j)) = (x_{n_1 - i + 1}, y_{n_2 - j + 1})$ is an opposition of
    $P_{n_1} \square P_{n_2}$ and so $P_{n_1} \square P_{n_2}$ is opposable.
    Since $P_{n_1} \square P_{n_2}$ is opposable, $\square^k_{i=1} P_{n_i}$ is
    opposable by \cref{cor: graph products preserve opposable}. 

    Suppose each $n_i$ is odd. Note that any path on an odd number of vertices
    satisfies the conditions of Theorem \ref{thm: a.o. graphs with automorphism
    Cartesian product is a.o.}. Thus the Cartesian product of any two paths
    also satisfies the conditions of Theorem \ref{thm: a.o. graphs with
    automorphism Cartesian product is a.o.}. By induction, $\square^k_{i=1}
    P_{n_i}$ is almost opposable and satisfies the conditions of Theorem
    \ref{thm: a.o. graphs with automorphism Cartesian product is a.o.}.

    Now suppose there is exactly one $n_i$ that is odd. Without loss of
    generality, assume $n_k$ is odd. The graph $\square^k_{i=1} P_{n_i}$ is
    isomorphic to $(\square^{k-1}_{i=1} P_{n_i}) \square (P_k)$. By the
    argument made in the case where $n_i$ is odd for all $i$,
    $\square^{k-1}_{i=1} P_{n_i}$ satisfies the conditions of Theorem \ref{thm:
    a.o. graphs with automorphism Cartesian product is a.o.}. Therefore,
    $(\square^{k-1}_{i=1} P_{n_i}) \square (P_k)$ is almost opposable. 
\end{proof}

Note that the condition in Theorem \ref{thm: a.o. graphs with automorphism
Cartesian product is a.o.} of being compatibly almost opposable is necessary to
guarantee that the Cartesian product yields an almost opposable graph. To see
this, let $G$ be the graph in Figure \ref{fig: a.o. auto. strong prod.
counterexample} with the vertices labelled in the same way. Recall that we
showed $G$ is almost opposable but does not have any nontrivial automorphisms.
We claim that the graph $G\square P_3$ is not almost opposable. 

Let $v_1,v_2,v_3$ be the vertices of $P_3$. Suppose for a contradiction
that $G\square P_3$ is almost opposable. Let $S$ be an admissible set such
that $G\square P_3 - S$ is opposable. If $x\in V(G)$ and $(x,v_1), (x,v_3)
\notin S$ then any automorphism of $G\square P_3 -S$ is forced to map
$(x,v_1)$ to $(x,v_3)$. For any given $x\in V(G)$, since the only vertex
adjacent to both $(x,v_1)$ and $(x,v_3)$ is $(x,v_2)$, if $(x,v_1), (x,v_3)
\notin S$ then any automorphism of $G\square P_3 - S$ is forced to map
$(x,v_2)$ to itself. Therefore, for $G-S$ to be opposable either $S$
contains all vertices in $V(G.v_2)$ or $S$ contains vertices that form a
set
\[
    \{(x,v_{i_x})\mid\text{$x\in V(G)$ and for each $x\in V(G)$, either $i_x =
    1$ or $i_x = 2$}\}.
\]
Neither of these scenarios is possible since $S$ is an admissible set and so
$G\square P_3$ is not almost opposable.

\section{Peaceable Queens Game}
\label{sec: peaceable queens game}

Consider a game played on a chessboard where two players take turns placing a
queen on an unoccupied space of the board. The player to go first always places
white queens while the player that moves second always places black queens.
Neither player is allowed to place a queen that is within attacking range of
any of the queens previously placed by their opponent. If one of the players is
unable to place a queen on the chessboard, then that player loses the game. We
call this the \emph{Peaceable Queens game} and in this section we explore the
outcomes of the game for $n\times m$ chessboards.

Denote $\textsc{pq}(n,m)$ for the Peaceable Queens Game on an $n\times m$ grid.
We can also define similar games using other chess pieces: we write
$\textsc{pr}(n,m)$, $\textsc{pn}(n,m)$, $\textsc{pb}(n,m)$, and
$\textsc{pk}(n,m)$ for the Peaceable Rooks Game, the Peaceable Knights Game,
the Peaceable Bishops Game, and the Peaceable Kings Game respectively. 

Each of these peaceable chess piece games is isomorphic to particular
\textsc{snort} positions. Consider an $n\times m$ array of vertices labelled
$(x,y)$ where $0\leq x\leq n-1$ and $0\leq y\leq m-1$. The vertex $(x,y)$
corresponds to the square $(x,y)$ on the chessboard. Two vertices $(x_1, y_1)$
and $(x_2, y_2)$ will be adjacent in the array if the corresponding squares on
the chessboard are within attacking range of each other based on the chess
piece being considered. For example, the vertices $(0,0)$ and $(4,4)$ are
adjacent to each other when considering Queens but not when we are considering
Knights. We will refer to the graph generated by considering Queens as the
\emph{Queen's grid}. We similarly define the \emph{Bishop's grid}, the
\emph{Knight's grid}, the \emph{Rook's grid} and the \emph{King's grid}. The
Peaceable Queens game is isomorphic to \textsc{snort} played on the Queen's
grid and similar statements can be made for the other grids. To determine the
outcomes of each of the peaceable games, we will show that corresponding grid
graphs are either opposable or almost opposable. For each of the proofs we will
label the vertices the same way as above.

We begin with the Peaceable Kings game as its outcomes follow directly from
Corollary \ref{cor: strong grid is a.o.}.

\begin{theorem}
    \label{thm: PK outcomes}
    The second player never wins $\textsc{pk}(n,m)$. 
\end{theorem}

\begin{proof}
    The $n\times m$ King's grid is isomorphic to $P_n \boxtimes P_m$. By
    Corollary \ref{cor: strong grid is a.o.}, all $k$-dimensional strong grids
    are almost opposable, and so every King's grid is almost opposable. By
    Lemma \ref{lem: a.o. implies 1st player win}, the first player wins
    \textsc{snort} on every King's grid. 
\end{proof}

For the Peaceable Knights, Bishops, and Rooks games, we completely determine
the outcomes by directly showing the corresponding grids are opposable or
almost opposable.

\begin{theorem}
    \label{thm: PN outcomes}
    The second player wins $\textsc{pn}(n,m)$ if and only if either $n$ or $m$
    are even.
\end{theorem}

\begin{proof}
    In the Knight's grid, two vertices $(x_1, y_1)$ and $(x_2, y_2)$ are
    adjacent if and only if either 
    \begin{itemize}
        \item $x_2 = x_1 \pm 1$ and $y_2 = y_1 \pm 2$ or 
        \item $x_2 = x_1 \pm 2$ and $y_2 = y_1 \pm 1$.
    \end{itemize}

    Suppose, without loss of generality, that $n$ is even. Define $f: V(G_K)
    \to V(G_K)$ by $f((x,y)) = (n-x-1, y)$. Note that $f$ is an automorphism of
    order two. Since $n$ is even, $f$ has no fixed points. Furthermore, since
    $y\neq y \pm 1$ and $y\neq y \pm 2$, $f((x,y))$ and $(x,y)$ are not
    adjacent. Therefore $f$ is an opposition of $G_K$. Thus by Lemma \ref{lem:
    opposable implies second player wins snort}, the second player wins
    \textsc{snort} on $G_K$ and so the second player wins $\textsc{pn}(n,m)$. 

    Instead suppose both $n$ and $m$ are odd. Let $g: V(G_K - (\frac{n-1}{2},
    \frac{m-1}{2})) \to V(G_K - (\frac{n-1}{2}, \frac{m-1}{2}))$ be defined by
    $g((x,y)) = (n-x-1, m-y-1)$. We claim that $g$ is an opposition of $G_K -
    (\frac{n-1}{2}, \frac{m-1}{2})$. Note that $g$ is an automorphism of order
    two with no fixed points. Suppose for a contradiction that $(x,y)$ and
    $g((x,y))$ were adjacent. Then either $n-x-1 = x\pm 1$ or $n-x-1 = x\pm 2$.
    Thus either $n-2x-1 = \pm 1$ or $n-2x-1 \pm 2$. Since $2x-1$ is odd and $n$
    is odd, $n-2x-1$ is odd. So $n-2x-1 \neq \pm 2$ and consequently $n-2x-1 =
    \pm 1$. This forces $m-y-1 = y\pm 2$ which is a contradiction since $m$ and
    $2y-1$ are both odd. Therefore $g((x,y))$ and $(x,y)$ are not adjacent for
    any $0\leq x\leq n-1$ and $0\leq y\leq m-1$. Since $g$ is an opposition of
    $G_K - (\frac{n-1}{2}, \frac{m-1}{2})$, $G_K$ is almost opposable. By Lemma
    \ref{lem: a.o. implies 1st player win}, the first player wins
    \textsc{snort} on $G_K$ and so $\textsc{pn}$ is first player win when both
    $n$ and $m$ are odd.
\end{proof}

\begin{theorem}
    \label{thm: PB outcomes}
    The second player wins $\textsc{pb}(n,m)$ if and only if either $n$ or $m$
    are even.
\end{theorem}

\begin{proof}
    In the Bishop's grid $G_B$, two vertices $(x_1, y_1)$ and $(x_2, y_2)$ are
    adjacent if and only if $|x_1 - x_2| = |y_1 - y_2|$. 

    Without loss of generality, suppose $n$ is even. We define $f: V(G_B) \to
    V(G_B)$ by $f((x,y)) = (n-x-1, y)$. Note that $f$ is a automorphism of
    order two. Since $n$ is even, $x \neq n-x-1$ and so $f$ has no fixed
    points. Suppose for a contradiction that for some $0\leq x\leq n-1$ and
    $0\leq y\leq m-1$, $(x,y)$ and $(n-x-1, y)$ are adjacent to each other.
    Then $|x-(n-x-1)| = |y-y|$ and so $n=2x+1$ which contradicts $n$ being
    even. Therefore $f$ is an opposition. Thus by Lemma \ref{lem: opposable
    implies second player wins snort}, the second player wins on $G_B$. 

    Now suppose that both $n$ and $m$ are odd. We claim that $G_B$ is almost
    opposable. Consider the subgraph $H = G_B - N[(\frac{n-1}{2},
    \frac{m-1}{2})]$. We define $f:V(H) \to V(H)$ by $f((x,y)) = (n-x-1,
    m-y-1)$ and claim that $f$ is an opposition of $H$. Note that $f$ is an
    automorphism of order two with no fixed points. Let $(x,y) \in V(H)$ and
    suppose for a contradiction that $(x,y)$ is adjacent to $f((x,y))$. Then
    $|x - (n-x-1)| = |y - (m-y-1)|$ and so $|2x+1-n| = |2y+1-m|$. Thus either 
    \begin{equation}
        \label{eq: Bishops 1st}
        2x+1-n = 2y+1-m
    \end{equation}
    or 
    \begin{equation}
        \label{eq: Bishops 2nd}
        2x+1-n = m-(2y+1).
    \end{equation}
    If Equation \eqref{eq: Bishops 1st} holds then rearranging yields $x-
    \frac{n-1}{2} = y - \frac{m-1}{2}$. Thus $|x- \frac{n-1}{2}| = |y -
    \frac{m-1}{2}|$ and so $(x,y) \in N[(\frac{n-1}{2}, \frac{m-1}{2})]$ which
    is a contradiction. If Equation \eqref{eq: Bishops 2nd} holds then
    rearranging yields $x-\frac{n-1}{2} = -\left( y-\frac{m-1}{2} \right)$.
    Thus $|x- \frac{n-1}{2}| = |y - \frac{m-1}{2}|$ and again $(x,y) \in
    N[(\frac{n-1}{2}, \frac{m-1}{2})]$. Therefore for all $v\in V(H)$, $f(v)$
    is not adjacent to $v$ and so $f$ is an opposition of $H$. So $G_B$ is
    almost opposable and the first player wins \textsc{snort} on $G_B$ by Lemma
    \ref{lem: a.o. implies 1st player win}.
\end{proof}

\begin{theorem}
    \label{thm: PR outcomes}
    The second player wins $\textsc{pr}(n,m)$ if and only if both $n$ and $m$
    are even.
\end{theorem}

\begin{proof}
    In the Rook's grid $G_R$, two vertices $(x_1, y_1)$ and $(x_2, y_2)$ are
    adjacent if and only if either $x_1 = x_2$ or $y_1 = y_2$. 

    Suppose both $n$ and $m$ are even. Define $f:V(G_R) \to V(G_R)$ by
    $f((x,y)) = (n-x-1, m-y-1)$ and note that $f$ is an automorphism of order
    two with no fixed points. Since $x=n-x-1$ and $y=m-y-1$ if and only if $n$
    and $m$ are odd, $f((x,y))$ is not adjacent to $(x,y)$ for all $(x,y) \in
    V(G_R)$. Therefore $f$ is an opposition and so by Lemma \ref{lem: opposable
    implies second player wins snort}, the second player wins \textsc{snort} on
    $G_R$. 

    Suppose instead that at least one of $n$ or $m$ are odd. We consider two
    cases, the first being exactly one of $n$ and $m$ are odd and the second
    being both $n$ and $m$ are odd. Without loss of generality, assume $n$ is
    odd and $m$ is even. Let $H = G_R - \{(\frac{n-1}{2}, y) \mid 0\leq y\leq
    m-1\}$. Note that $\{(\frac{n-1}{2}, y) \mid 0\leq y\leq m-1\} \subseteq
    N_{G_R}[(\frac{n-1}{2}, y)]$ for any $0\leq y\leq m-1$. Define $f:V(H) \to
    V(H)$ by $f((x,y)) = (n-x-1, y)$. Note that $f$ is an automorphism of order
    two with no fixed points. For any $0\leq x\leq n-1$, $x=n-x-1$ if and only
    if $x= \frac{n-1}{2}$. Therefore for all $(x,y) \in V(H)$, $(x,y)$ and
    $f((x,y)$ are not adjacent and so $f$ is an opposition. Thus by Lemma
    \ref{lem: a.o. implies 1st player win}, the first player wins
    \textsc{snort} on $G_R$. 

    Now suppose that both $n$ and $m$ are odd. Consider the automorphism $g$ on
    $G_R - N[(\frac{n-1}{2}, \frac{m-1}{2})]$ defined by $g((x,y)) = (n-x-1,
    m-y-1)$. Note that $g$ is an automorphism of order two with no fixed
    points. Since $(x,y)$ and $(n-x-1, m-x-1)$ are adjacent if and only if
    either $x= \frac{n-1}{2}$ or $y=\frac{m-1}{2}$, $g$ is an opposition of
    $G_R - N[(\frac{n-1}{2}, \frac{m-1}{2})]$. Therefore $G_R$ is almost
    opposable and so the first player wins \textsc{snort} on $G_R$ by Lemma
    \ref{lem: a.o. implies 1st player win}.
\end{proof}

Next, we determine the outcome of the Peaceable Queens game played on a board
with either an odd number of rows or an odd number of columns. 

\begin{theorem}
    \label{thm: PQ outcomes for odd x m boards}
    If either $n$ or $m$ are odd then the first player wins $\textsc{pq}(n,m)$.
\end{theorem}

\begin{proof}
    In the Queen's grid $G_Q$, two vertices $(x_1, y_1)$ and $(x_2, y_2)$ are
    adjacent if and only if either $x_1 = x_2$, $y_1 = y_2$, or $|x_1 - x_2| =
    |y_1 - y_2|$. Note that these conditions for adjacency are the same as
    testing for adjacency in the Rook's grid and the Bishop's grid
    simultaneously. 

    To show that the first player wins on $G_Q$ when either $n$ or $m$ are odd,
    we will show that $G_Q$ is almost opposable when either $n$ or $m$ are odd.
    Without loss of generality, suppose $n$ is odd and $m$ is even. Let $S =
    \{(\frac{n-1}{2},y) \mid 0\leq y\leq m\}$ and let $f: V(G_Q - S) \to V(G_Q
    - S)$ be defined by $f((x,y)) = (n-x-1, m-y-1)$. By the work done in the
    proofs of Theorem \ref{thm: PB outcomes} and Theorem \ref{thm: PR
    outcomes}, $(x,y)$ and $(n-x-1, m-y-1)$ are not adjacent in both the
    Bishop's grid and the Rook's grid for all $0\leq x\leq n-1$ and $0\leq
    y\leq m-1$. Therefore $(x,y)$ and $f((x,y))$ are not adjacent in the
    Queen's grid and so $f$ is an opposition of $G_Q - S$. 

    Now suppose both $n$ and $m$ are odd. Let $H = G_Q - N[(\frac{n-1}{2},
    \frac{m-1}{2})]$ and let $g: V(H) \to V(H)$ be defined by $g((x,y)) =
    (n-x-1, m-x-1)$. By the work done in the proofs of  Theorem \ref{thm: PB
    outcomes} and Theorem \ref{thm: PR outcomes}, $(x,y)$ and $g((x,y))$ are
    not adjacent for any $0\leq x\leq n-1$ and $0\leq y\leq m-1$. So $g$ is an
    opposition of $H$.

    Therefore, when at least one of $n$ or $m$ is odd, the Queen's grid is
    almost opposable and thus the first player wins $\textsc{pq}(n,m)$ by
    \cref{lem: a.o. implies 1st player win}.
\end{proof}

We believe that determining who wins the Peaceable Queens Game on $2n \times
2m$ chessboards cannot be done with the techniques used in this paper. Without
loss of generality, assume $n\leq m$. To see that the $2n \times 2m$ Queen's
grid is not opposable, suppose for a contradiction that $\alpha$ is an
opposition of the Queen's grid. The Queen's grid has four vertices of minimum
degree; $(0,0)$, $(2n-1,0)$, $(0,2m-1)$, and $(2n-1,2m-1)$. Thus $\alpha$ maps
these four vertices to each other. In the case where $n=m$, these four vertices
are all adjacent to each other and so we already have a contradiction. From
here we assume that $n < m$. Since $(0,0)$ is adjacent to $(2n-1,0)$ and
$(0,2m-1)$, $\alpha((0,0)) = (2n-1, 2m-1)$ and thus $\alpha((2n-1, 0)) =
(0,2m-1)$. To preserve adjacency, for each $1\leq x\leq 2n-1$ and $1\leq y\leq
2m-1$, $\alpha$ is forced to map $(x,0)$ to $(2n-x-1, 2m-1)$ and $(0,y)$ to
$(2n-1, 2m-y-1)$. Consequently, $\alpha$ is forced to map $(x,y)$ to $(2n-x-1,
2m-y-1)$ for all $0\leq x\leq 2n-1$ and $0\leq y\leq 2m-1$.

Let $A = \{(x, m-n+x) \mid 0\leq x\leq 2n-1\}$ and $B= \{(x, m+n-x-1) \mid
0\leq x\leq 2n-1\}$. Informally, $A$ and $B$ are each sets of $2n$ vertices and
the vertices of $A\cup B$ form the diagonals containing the four centre
vertices of the grid. Note that all of the vertices in $A$ are adjacent to each
other and all of the vertices in $B$ are adjacent to each other. However, by
the above argument, every vertex in $A$ is mapped by $\alpha$ to another vertex
in $A$. Similarly, $\alpha$ maps vertices in $B$ to other vertices in $B$. This
is a contradiction and so the $2n \times 2m$ Queen's grid is not opposable.

While it is easy to show that $2n \times 2m$ Queen's grids are not opposable,
proving that they are not almost opposable is more challenging. It appears that
they are not almost opposable (except in simple cases such as $n=m=1$), however
we leave proving this as an open problem.

\begin{problem}
    When is the $2n \times 2m$ Queen's grid almost opposable for $n,m \in
    \mathbb{Z}^+$?
\end{problem}

\section{Further Directions}
\label{sec: further directions}

Instead of having black and white queens like in the Peaceable Queens Game, we
could consider an impartial game where all queens are the same colour and the
two players are not allowed to place a queen that is within attacking range of
any queen that is already on the chessboard. Noon and Van Brummelen
\cite{noon:surreal, noon.van-brummelen:non-attacking} introduced and studied
this game on $n\times n$ chessboards and Brown and Ladha
\cite{brown.ladha:exploring} explored variations. Noon and Van Brummelen
\cite{noon:surreal, noon.van-brummelen:non-attacking} showed that on all
$n\times n$ boards for $1\leq n\leq 9$, the first player wins the impartial
game. However, the second player wins on the $10\times 10$ board. This leads
naturally to the question of whether the second player can ever win the
Peaceable Queens game on a $2n \times 2m$ board.

\begin{problem}
    Is it ever possible for the second player to win $\textsc{pq}(2n,2m)$?
\end{problem}

In \cite{beals.chang.ea:automorphisms}, it is shown that the function that
counts the number of automorphisms of a graph is $\exp(O(\sqrt{n\log
n}))$-enumerable. 
\begin{problem}
    What can be said about the enumerability of counting the number of
    oppositions of a graph? What about the number of ways a graph can be almost
    opposable?
\end{problem}
We note that counting perfect matchings (in possible service of counting
opposition matchings) is $\#\P$-complete \cite{valiant:complexity}.

One can see in multiple ways that the problem of deciding whether a given graph
is opposable is in $\NP$. Perhaps easiest is to non-deterministically guess an
opposition matching on the complement, and check that every pair of edges in
the matching induces a valid subgraph. Another way is to construct an auxiliary
graph $G'$ where edges in $G$ correspond to vertices in $G'$, and vertices in
$G'$ are adjacent if and only if the corresponding edges in $G$ induce a valid
subgraph within an opposition matching. Then opposition matchings correspond to
$|V(G)|/2$-cliques in $G'$, and $\lang{CLIQUE}$ is an $\NP$-complete problem.
Similarly, almost-opposability is also in $\NP$, as one can
non-deterministically guess a partial neighbourhood and opposition matching on
the resulting graph, then verify the validity of these guesses. 

The fact that (almost) opposability is in $\NP$ contrasts with the problem of
identifying whether a given \textsc{snort} position is winning being
$\PSPACE$-complete \cite{schaefer:complexity}. To correct this dissonance, we
observe that a \textsc{snort} strategy based around (almost) opposability is
one where each move relies on only (respectively at most) the previous move.
However, when reducing from $\lang{QBF}$ to show $\PSPACE$-hardness, strategies
may observe the entire history of plays. As a final note on complexity, it is
known that deciding if a graph has a perfect matching is solvable in $P$ time
by contracting odd cycles to a point and using the Ford--Fulkerson algorithm to
solve a max-flow problem across the resulting bipartite graph
\cite{edmonds:paths}. However, these natural techniques do not guarantee that
the resulting matchings satisfy the conditions of opposition matchings. Further
complicating the issue, graphs such as $P_2\square P_{2n}$ yield both
opposition matchings (lift opposition matching of $P_2$) and non-opposition
perfect matchings (lift perfect non-opposition matching of $P_{2n}$). We leave
open the issue of complexity of identifying if a graph is (almost) opposable.

Next, we mention two combinatorial games played on graphs where determining the
structural properties of first player winning graphs and second player winning
graphs remains open. 

Impartial \textsc{snort}, also written as \textsc{isnort}, is played with a
nearly identical ruleset to \textsc{snort} with the only difference being that
both players have access to both colours instead of just one colour each.
Despite the similarities of \textsc{snort} and \textsc{isnort}, not all of our
results hold for \textsc{isnort}. Uiterwijk \cite{uiterwijk:solving2} showed
that the second player wins \textsc{isnort} on $P_n \square P_m$ if and only if
either $n$ or $m$ is even. In the case where at most one of $n$ and $m$ are
even, we know the first player wins \textsc{snort} on $P_n \square P_m$ by
\cref{cor: k-dim Cartesian grids}.

\begin{problem}
    For which graphs does the second player win \textsc{snort} and the first
    player win \textsc{isnort}?
\end{problem}

A combinatorial game for digraphs called \textsc{digraph placement} was
recently introduced in \cite{clow.mckay:digraph}. In \textsc{digraph
placement}, the vertices of a given digraph are coloured either red or blue.
The two players, Left and Right, take turns deleting vertices along with their
out-neighbourhoods. Left is only allowed to delete blue vertices and Right is
only allowed to delete red vertices. Just as we did with \textsc{snort} in this
paper, we can ask questions regarding the structural properties of digraphs
where a player wins \textsc{digraph placement} by using a ``mirroring''
strategy. 

\begin{problem}
    What structural properties for digraphs yield simple winning strategies in
    \textsc{digraph placement} for either the first or second player?
\end{problem}

\bibliographystyle{plainurl}
\bibliography{bib}

\end{document}